\documentclass[11pt]{article}
\usepackage[nohead,margin=1.0in]{geometry}
\usepackage{amssymb, amsmath, amsthm,amsfonts}
\usepackage{graphicx,epsfig}
\usepackage[tight]{subfigure}
\usepackage[para,online,flushleft]{threeparttable}
\usepackage{cite}
\usepackage{times}
\usepackage{algpseudocode}
\usepackage{float}
\usepackage{appendix}
\usepackage{color}
\usepackage{epstopdf}
\usepackage{mathrsfs}
\usepackage{algorithm}
\usepackage{booktabs}
\usepackage{verbatim}
\usepackage[colorlinks,linktocpage,linkcolor=blue]{hyperref}

\numberwithin{equation}{section}
\newtheorem{theorem}{Theorem}[section]
\newtheorem{remark}{Remark}[section]

\newtheorem{lemma}{Lemma}[section]
\newtheorem{corollary}{Corollary}[section]

\newtheorem{assumption}[theorem]{Assumption}
\newtheorem{example}{Example}[section]
\allowdisplaybreaks

\def\P01{P_{\mathcal{A}}}
\def\II{(\Omega)}
\def\d{\mathrm{d}}

\title{Hybrid Neural-Network FEM Approximation of Diffusion Coefficient in Elliptic and Parabolic Problems\thanks{The work of B. Jin is supported by UK EPSRC grant EP/T000864/1 and EP/V026259/1, and a start-up fund from The Chinese University of Hong Kong. The work of  Z. Zhou is supported by Hong Kong Research Grants Council (15304420) and an internal grant of Hong Kong Polytechnic University (Project ID: P0038888, Work Programme: ZVX3).}}

\author{Siyu Cen\thanks{Department of Applied Mathematics, The Hong Kong Polytechnic University, Kowloon, Hong Kong, P.R. China. (\texttt{siyu2021.cen@connect.polyu.hk; zhizhou@polyu.edu.hk})}
\and Bangti Jin\thanks{Department of Mathematics, The Chinese University of Hong Kong, Shatin, New Territories, Hong Kong, P.R. China (\texttt{bangti.jin@gmail.com, b.jin@cuhk.edu.hk}).}
\and Qimeng Quan\thanks{School of Mathematics and Statistics, Wuhan University, Wuhan 430072, P. R. China (\texttt{quanqm@whu.edu.cn})}
\and Zhi Zhou\footnotemark[2]}

\begin{document}

\maketitle

\begin{abstract}
In this work we investigate the numerical identification of the diffusion coefficient in elliptic and parabolic problems using neural networks. The numerical scheme is based on the standard output least-squares formulation where the Galerkin finite element method (FEM) is employed to approximate the state and neural networks (NNs) act as a smoothness prior to approximate the unknown diffusion coefficient. A projection operation is applied to the NN approximation in order to preserve the physical box constraint on the unknown coefficient. The hybrid approach enjoys both rigorous mathematical foundation of the FEM and inductive bias / approximation properties of NNs. We derive \textsl{a priori} error estimates in the standard $L^2(\Omega)$ norm for the numerical reconstruction, under a positivity condition which can be verified for a large class of problem data.
The error bounds depend explicitly on the noise level, regularization parameter and discretization parameters (e.g., spatial mesh size, time step size, and depth, upper bound and number of nonzero parameters of NNs).
We also provide extensive numerical experiments, indicating that the hybrid method is very robust for large noise when compared with the pure FEM approximation.
\end{abstract}

\section{Introduction}
In this work, we study the inverse problem of recovering a space-dependent diffusion coefficient in elliptic and parabolic problems from one internal measurement using neural networks. Let $\Omega \subset\mathbb{R}^d\,(d=1,2,3)$ be a convex polyhedral domain with a boundary $\partial\Omega$. Consider the following elliptic problem
\begin{equation}\label{equ:elliptic problem}
	\left\{\begin{aligned}
		-\nabla\cdot(q\nabla u) &= f, \ &\mbox{in}&\ \Omega, \\
		u&=0, \ &\mbox{on}&\ \partial\Omega,
	\end{aligned}\right.
\end{equation}
where $f$ is a known source. The diffusion coefficient $q$ belongs to the admissible set
\begin{equation*}
	\mathcal{A} = \{q\in H^1(\Omega): c_0\leq q(x)\leq c_1 \mbox{ a.e. in } \Omega\},
\end{equation*}
with the constants $0<c_0<c_1<\infty$ being the lower and upper bounds on the diffusivity.
Below we use the notation $u(q)$ to indicate the dependence of the solution $u$ to problem \eqref{equ:elliptic problem} on the coefficient $q$. Further, we are given the noisy observational data $z^\delta$ in the domain $\Omega$:
\begin{equation*}
z^\delta(x) = u(q^\dagger)(x) + \xi(x), \quad x\in \Omega,
\end{equation*}
where $u(q^\dagger)$ denotes the exact data (for the exact coefficient $q^\dagger$), and $\xi$ denotes the noise. The data $z^\delta$ has an accuracy
$\delta=\|u(q^\dagger)-z^\delta\|_{L^2(\Omega)}$.
The inverse problem is to identify the diffusion coefficient $q^\dagger$ from $z^\delta$.  It arises naturally in many physical processes, e.g., in hydrology, where the parameter $q^\dag$ represents hydraulic diffusivity (or transmissivity in the 2D case) in the study of a confined inhomogeneous aquifer \cite{FrindPinder:1973,Yeh:1986}.

Due to excellent approximation property of neural networks (NNs) and recent algorithmic innovations, many methods based on NNs have been devised and have demonstrated impressive empirical performance on a variety of PDE inverse problems (see \cite{TanyuMaass:2022} for a recent overview). One prominent approach within the class is physics-informed neural networks (PINNs) \cite{raissi2019physics}. In the context of inverse problems, the idea is to minimize a PDE residual functional, and then to enforce both consistency with observational data via a suitable data-fitting functional and \textit{a priori} regularity assumption on the unknown via a suitable penalty. The unknowns are then approximated via NNs, and the resulting loss is trained to yield an approximation.
The theoretical analysis of neural PDE solvers for direct problems is still at an early stage, when compared with more conventional numerical methods, e.g., finite element methods (FEMs). This has greatly hindered the mathematical analysis of relevant inverse solvers.
To have the best of both approaches, one natural idea is to combine neural networks (NNs) with FEM.

In this work, we study the hybrid NN-FEM approach for recovering the unknown coefficient $q$ in problem \eqref{equ:elliptic problem} (and also the parabolic case in \eqref{equ:parabolic problem}), and provide an analysis on the numerical approximation. We contribute in the following three aspects. First, we develop a novel reconstruction formulation by incorporating the projection operator, which automatically guarantees the well-posedness of the discrete formulations. Second, we derive the $L^2(\Omega)$ error estimates on the NN approximation $q_\theta^*$ for both inverse elliptic and parabolic problems, under mild conditions on the problem data ($u_0$, $f$, $q$ and $\Omega$), cf. Theorems \ref{thm:error-ellip} and \ref{thm: error in para} for the elliptic and parabolic cases, respectively. The error bounds depend explicitly on the approximation accuracy  $\epsilon$ of the NN, discretization parameters ($h$ and $\tau$), the noise level $\delta$ and the regularization parameter $\gamma$. The overall argument relies heavily on a suitable positivity condition, cf. \eqref{Cond: P} and \eqref{Cond: P para}.
Third and last, in the context of hybrid solvers, quadrature errors are inevitable, due to the presence of the NN function in various integrals. We derive a useful $L^2(\Omega)$ bound depending on the NN architecture (e.g., width and maximum bound), cf. Theorems \ref{thm:error-ellip-q} and \ref{thm: error in para, quad}. The technical proofs rely on  smoothness properties of NNs and the structure of the finite element space. To the best of our knowledge, these results are new and provide theoretical foundations for using the hybrid formulation for solving PDE inverse problems.

Now we review existing works on the hybrid numerical approximation of the concerned inverse problem. The study on identifying the diffusion coefficient using hybrid discretization is firstly proposed in \cite{berg2021neural}. Berg and Nystr\"{o}m \cite{berg2021neural} proposed the hybrid discretization strategy, i.e., discretize the unknown coefficient and the state by NNs and Galerkin FEM, respectively, in an unregularized functional (i.e., data-fitting only), discussed in detail the training of the NN using the adjoint technique \cite{Cea:1986}, and presented extensive numerical experiments with recovering the diffusion coefficient in second-order elliptic problems from full and partial interior data.
The goal is to use NNs as an implicit smoothness prior. Since the objective is unregularized, the approach solely relies on the use of shallow one-layer NNs (with very few neurons) to achieve sufficient regularization (in a manner similar to the classical projection methods).
The latter restricts the admissible coefficient to a very low-dimensional manifold.  Later, Mitusch et al \cite{mitusch2021hybrid} extended the hybrid formulation to both stationary and transient as well as linear/nonlinear PDEs, and discussed the crucial role of a proper penalty term, including the $H^1(\Omega)$ penalty, in order to stabilize the training process.  In both works \cite{berg2021neural,mitusch2021hybrid}, the authors  provided extensive
 numerical experiments to show the practicability of the hybrid NN-FEM method with one-layer NNs, by comparing the approach with the more conventional FEM discretization.
The extensive numerical studies in the works \cite{berg2021neural,mitusch2021hybrid} show that the approach does have certain merits and holds some potentials for PDE inverse problems.
 Huang et al \cite{HuangDarve:2020jcp} proposed the hybrid approach to learn constitutive relations from indirect observations, where the physical parameters / laws are expressed by NNs whereas the state is discretized using FEM. In addition, the authors provided a software framework for a variety of problems, and discussed additional functionality, e.g., uncertainty quantification. The formulation in this work differs in using a projection operator to ensure the box constraint.

Neural inverse PDE solvers have attracted significant attention from both theoretical and numerical aspects (see \cite{TanyuMaass:2022} for a recent overview).
However, so far there are still very few error bounds on discrete approximations of the neural inverse solvers, including hybrid NN-FEM. This is attributed to the high complexity of the discretization strategy, nonlinearity of the forward map for the concerned inverse problem and strong nonconvexity of the regularized functional. One exception is \cite{MishraMolinaro:2022ima},  where the authors derived an abstract error estimate for the PINNs solving linear inverse problems using the concept of conditional stability. Jin et al \cite{jin2022imaging} derived a generalization bound on the functional, but not on the recovered conductivity. Kaltenbacher and Nguyen \cite{KaltenbacherNguyen:2022} studied approximating a nonlinearity in a parabolic model with NNs and analyzed the convergence for Tikhonov regularization and Landweber method. In sharp contrast, the theory for the Galerkin FEM approximation is relatively well understood. Indeed, in a series of works \cite{wang2010error,jin2021error,jin2022convergence}, several researchers have investigated the standard Galerkin FEM discretization for both diffusion coefficient and the state, and establish $L^2(\Omega)$ error bounds on the numerical approximation. The overall proofs include conditional stability argument and suitable positivity conditions. In this work, we combine these tools with recent NN approximation theory, and derive first error bounds for hybrid discretization. It substantially expands the scope of the numerical analysis of neural inverse PDE solvers, and represents an important step towards analyzing fully neural inversion schemes.

The rest of the paper is organized as follows. In Sections \ref{sec:elliptic} and \ref{sec:parabolic}, we establish the $L^2(\Omega)$ error bounds of the hybrid NN-FEM approximation for elliptic and parabolic cases with or without numerical quadrature.
In Section \ref{sec: numerics}, we describe the algorithmic details of the approaches and present several numerical experiments to complement the theoretical results. Throughout, for any $m\geq0$ and $p\geq1$, we denote by $W^{m,p}(\Omega)$ and $W^{m,p}_0(\Omega)$ the standard Sobolev spaces of order $m$, equipped with the norm $\|\cdot\|_{W^{m,p}(\Omega)}$ and the semi-norm $|\cdot|_{W^{m,p}(\Omega)}$. We also write $H^{m}(\Omega)$ and $H^{m}_{0}(\Omega)$ with the norm $\|\cdot\|_{H^m(\Omega)}$ if $p=2$ and write $L^p(\Omega)$ with the norm $\|\cdot\|_{L^p(\Omega)}$ if $m=0$. We use $(\cdot,\cdot)$ to denote the $L^2(\Omega)$ inner product.  We denote by $c$ a generic  constant not necessarily the same at each occurrence but it is always independent of the approximation accuracy $\epsilon$ of the NN (to the exact coefficient $q^\dag$), discretization parameters $h$ and $\tau$, noise level $\delta$ and regularization parameter $\gamma$.

\section{Preliminaries}

\subsection{Neural networks}
In this work, we employ fully connected feedforward neural networks. Let $L\in\mathbb{N}$ be the depth of a neural network (NN) and $\{d_\ell\}_{\ell=0}^L\subset\mathbb{N}$ be a sequence of integers, with $d_0=d$ and $d_L=1$, $d_\ell$ the number of neurons in the $\ell$th layer of the NN. Then the realization of the NN from $\Omega\subset\mathbb{R}^d$ to $\mathbb{R}$ is defined by
\begin{equation}\label{equ: network realization}
	\mbox{NN realization}\quad\left\{\begin{aligned}
		&v^{(0)}=x,\quad x\in\Omega,\\
		&v^{(\ell)}=\rho(A^{(\ell)}v^{(\ell-1)}+b^{(\ell)}),\quad \mbox{for }\ell=1,2,\cdots,L-1,\\
		&v:=v^{(L)}=A^{(L)}v^{(L-1)}+b^{(L)},
	\end{aligned}\right.
\end{equation}
where $\rho:\mathbb{R}\to\mathbb{R}$ is a nonlinear activation function and applied componentwise to a vector. Throughout, we take $\rho\equiv\tanh$: $x\to\frac{e^x-e^{-x}}{e^x+e^{-x}}$. $A^{(\ell)}\in\mathbb{R}^{d_\ell\times d_{\ell-1}}$ and $b^{(\ell)}\in\mathbb{R}^{d_{\ell}}$ are weight matrices and bias vectors at the $\ell$-th layer of the NN. The width $W$ of the NN is defined by $W:=\max_{\ell=0,\dots,L}d_\ell$. We denote the NN parametrization by $\theta=\{(A^{(\ell)},b^{(\ell)})\}_{\ell=1}^L\in \prod_{\ell=1}^L(\mathbb{R}^{d_\ell\times d_{\ell-1}}\times\mathbb{R}^{d_{\ell}})$.  The following approximation property holds \cite[Proposition 4.8]{guhring2021approximation}.
\begin{lemma}\label{lem:tanh-approx}
Let $s\in\mathbb{N}_0$ and $p\in[1,\infty]$ be fixed, and $v\in W^{k,p}(\Omega)$ with $k\geq  s+1$. Then for any $\epsilon>0$, there exists at least one $\theta\in\Theta$ with depth $\mathcal{O}\big(\log(d+k)\big)$ and total number of nonzero parameters $\mathcal{O}\big(\epsilon^{-\frac{d}{k-s-\mu (s=2)}}\big)$, where $\mu>0$ is arbitrarily small, such that the NN realization  $v_\theta$ of $\theta$ satisfies
\begin{equation}\label{lem: tanh approx}
     \|v-v_\theta\|_{W^{s,p}(\Omega)} \leq \epsilon.
\end{equation}
Moreover, the maximum norm of the weights in the NN is bounded by $\mathcal{O}(\epsilon^{-2-\frac{2(d/p+d+s+\mu(s=2))+d/p+d}{k-s-\mu (s=2)}})$.
\end{lemma}

We denote the set of NNs of depth $L$, the number of nonzero entries $N_\theta$, and maximum bound $R$ on the parameter vector $\theta$ by
\begin{equation*}
    \mathcal{N}(L,N_\theta,R) =: \{v_\theta \mbox{ is an NN with depth }L: \|\theta\|_{\ell^0}\leq N_\theta, \|\theta\|_{\ell^\infty}\leq R\},
\end{equation*}
where $\|\cdot\|_{\ell^0}$ and $\|\cdot\|_{\ell^\infty}$ denote the number of nonzero entries in and the maximum norm of, respectively, a vector. Further, for any $\epsilon>0$ and $p\geq 1$, we denote by $\mathfrak{P}_{p,\epsilon}$ the NN parameter set for the NN function class
\begin{equation*}
\mathcal{N}\Big(C\log(d+1), C\epsilon^{-\frac{d}{1-\mu}}, C \epsilon^{-2-\frac{2p+3d+3pd+2\mu}{p(1-\mu)}}\Big),
\end{equation*}
which will be used to approximate the coefficient $q$. We focus on two cases: $p=\max(2,d+\mu)$ (with small $\mu>0$) and $p=\infty$ for the cases without and with the quadrature error, respectively.

The next result bounds the tanh activation function $\rho$.
\begin{lemma}\label{lem:rho}
The following estimates hold
\begin{equation*}
    \|\rho\|_{L^\infty(\mathbb{R})}\leq 1, \quad  \|\rho'\|_{L^\infty(\mathbb{R})}\leq 1, \quad \|\rho''\|_{L^\infty(\mathbb{R})}\leq 1,\quad \|\rho'''\|_{L^\infty(\mathbb{R})}\leq 2
\end{equation*}
\end{lemma}
\begin{proof}
Clearly $\|\rho\|_{L^\infty(\mathbb{R})}\leq 1$. Next, using the definition of $\rho$, direct computation gives
\begin{align*}
    \rho'(x) 
    =1-\rho^2(x),\quad
    \rho''(x)  = -2\rho(x)(1-\rho^2(x)),\quad
    \rho'''(x)=(6\rho^2(x)-2)(1-\rho^2(x)).
\end{align*}
Thus the desired assertions follow directly.
\end{proof}
\subsection{Galerkin FEM}

In the standard Galerkin FEM \cite{Thomee:2006}, we divide the domain $\Omega$ into a quasi-uniform simplicial triangulation $\mathcal{T}_h$ with a mesh size $h$. Over $\mathcal{T}_h$, we define a conforming piecewise linear finite element space $X_h\subset H_0^1(\Omega)$ by
\begin{equation*}
X_{h}:=\{v_{h}\in H_0^1(\Omega):	v_{h}|_{K}\in P_1(K) ,\,\,\, \forall K\in\mathcal{T}_h\},
\end{equation*}
where $P_1(K)$ denotes the set of linear polynomials on the element $K$.
On the finite element space $X_{h}$,
we define the standard $L^2(\Omega)$-projection $P_h:L^2(\Omega)\to X_h$ by
\begin{equation*}
	(P_hv,\varphi_h)=(v,\varphi_h), \quad \forall v\in L^2(\Omega),\  \varphi_h\in X_h.
\end{equation*}
Then the operator $P_h$ is stable in both $L^2(\Omega)$ and $H_0^1(\Omega)$, and further the following approximation result holds \cite[Theorems 3.2 and 3.4]{BrambleXu:1991}: for $s=1,2$
\begin{equation}\label{inequ: L2 proj approx}
	\|v-P_hv\|_{L^2(\Omega)} +h \|\nabla(v-P_hv)\|_{L^2(\Omega)}\leq  ch^{s}\|v\|_{H^s(\Omega)}, \quad \forall v\in H^{s}(\Omega)\cap H_0^{1}(\Omega).
\end{equation}

\section{Elliptic inverse problem}\label{sec:elliptic}
Now we develop and analyze a novel hybrid NN-FEM approximation for the elliptic inverse problem.

\subsection{The regularized problem and its hybrid approximation}

To recover the diffusion coefficient $q$, we employ the standard regularized output least-squares formulation with an $H^1(\Omega)$ seminorm penalty, which amounts to minimizing the following objective:
\begin{equation}\label{equ:Tikhonov problem in ellip}
	\min_{q\in \mathcal{A}} J_{\gamma}(q)=\frac{1}{2}\|u(q)-z^{\delta}\|^2_{L^2(\Omega)}+\frac{\gamma}{2}\|\nabla q\|_{L^2(\Omega)}^2,
\end{equation}
with $u\equiv u(q)\in H^1_0(\Omega)$ subject to the following PDE constraint
\begin{equation}\label{equ:vari problem in ellip}	(q\nabla u,\nabla\varphi)=(f,\varphi), \quad\forall \varphi\in H^1_0(\Omega).
\end{equation}
A standard argument in calculus of variation shows the well-posedness of the regularized problem \eqref{equ:Tikhonov problem in ellip}-\eqref{equ:vari problem in ellip}: for any fixed $\gamma>0$, it has at least one global minimizer $q_\gamma^\delta$, which depends continuously on the data \cite{EnglKunischNeubauer:1989,ito2014inverse}. Moreover, as the noise level $\delta\to 0^+$, the sequence $\{q_\gamma^\delta\}_{\delta>0}$ of minimizers  converges to the exact coefficient $q^\dag$ in $H^1(\Omega)$, if the regularization parameter $\gamma$ is chosen properly in accordance with $\delta$ \cite{EnglKunischNeubauer:1989,ito2014inverse}.
In practice, the regularized problem has to be properly discretized, and this is often achieved using finite element / finite difference methods \cite{Richter:1981,Falk:1983,Zou:1998,HaoQuyen:2011}.

In this work, we employ an alternative discretization strategy: we approximate the coefficient $q$ using NNs, and the state $u$ using the Galerkin FEM. Note that NNs are globally defined, unlike compactly supported FEM basis functions. Hence, it is challenging to impose the box constraint of the admissible set $\mathcal{A}$ directly.
In order to preserve the box constraint of $\mathcal{A}$, we apply to the NN output a cutoff operation $P_{\mathcal{A}}: H^1(\Omega)\rightarrow \mathcal{A}$ defined by
\begin{equation}\label{eqn:P01}
\P01(v) = \min(\max(c_0,v),c_1).
\end{equation}
The operator $P_\mathcal{A}$ is stable in the following sense \cite[Corollary 2.1.8]{Ziemer:1989}
\begin{equation}\label{eqn:P01-stab}
\| \nabla \P01(v) \|_{L^p(\Omega)} \le \|\nabla v \|_{L^p(\Omega)},\quad \forall v \in W^{1,p}(\Omega),p\in[1,\infty],
\end{equation}
and moreover, for all $v\in \mathcal{A}$, there holds
\begin{equation}\label{eqn:P01-approx}
\|  \P01(w) - v \|_{L^p(\Omega)} \le \| w - v \|_{L^p(\Omega)},\quad \forall w \in L^{p}(\Omega),~p\in[1,\infty].
\end{equation}

Now we can formulate the hybrid NN-FEM approximation scheme as
\begin{equation}\label{equ:dis-min-ellip}
	\min_{\theta\in\mathfrak{P}_{p,\epsilon}} J_{\gamma,h}(q_\theta)=\frac{1}{2}\|u_h\big(\P01(q_\theta)\big)-z^{\delta}\|^2_{L^2(\Omega)}+\frac{\gamma}{2}\|\nabla q_\theta\|_{L^2(\Omega)}^2,
\end{equation}
where the discrete state $u_h\equiv u_h(\P01(q_\theta))\in X_h$ satisfies the following discrete variational problem
\begin{equation}\label{equ:dis-vari-ellip}	
	(\P01(q_\theta)\nabla u_h,\nabla\varphi_h)=(f,\varphi_h), \quad\forall \varphi_h\in X_h.
\end{equation}
The well-posedness of problem \eqref{equ:dis-min-ellip}-\eqref{equ:dis-vari-ellip} holds trivially true. Indeed, the uniform boundedness of the admissible set $\mathfrak{P}_{p,\epsilon}$ in a finite-dimension space implies the compactness of the parametrization set. Together with the continuity of the discrete forward map, the existence of a minimizer $\theta^*$ to problem \eqref{equ:dis-min-ellip}--\eqref{equ:dis-vari-ellip} follows by a standard argument. We denote its NN realization by $q_\theta^*$.

The hybrid formulation \eqref{equ:dis-min-ellip}-\eqref{equ:dis-vari-ellip} enjoys the following distinct features. First, the construction naturally preserves the box constraint, which is highly nontrivial to impose on the NN functions directly; Second, the resulting objective $J_{\gamma,h}(q_\theta)$ is differentiable with respect to the NN parameters $\theta$, which facilitates the training process by gradient type methods; Third, it is amenable with rigorous convergence analysis, i.e., \textsl{a priori} error estimates. In sum, it enjoys both rigorous mathematical foundation of the FEM and excellent inductive bias / approximation properties of NNs.

\begin{remark}
The formulation \eqref{equ:dis-min-ellip}--\eqref{equ:dis-vari-ellip} includes the operator $P_\mathcal{A}$, and uses $P_\mathcal{A}(q_\theta)$ to approximate the exact one $q^\dag$. It differs from the existing ones. Berg and Nystr\"{o}m \cite{berg2021neural} suggested the objective
\begin{equation*}
   \min_{\theta\in\mathfrak{P}_\epsilon} J_{\gamma,h}(q_\theta)=\frac{1}{2}\|u_h\big(q_\theta\big)-z^{\delta}\|^2_{L^2(\Omega)}+\frac{\gamma}{2}\|q_\theta\|_{L^2(\Omega)}^2,
\end{equation*}
where $\gamma\geq0$ is the regularization parameter. Their numerical evaluation focuses on $\gamma=0$, i.e., unregularized case, which necessitates the use of tiny NNs for approximating $q$, in order to avoid overfitting. The well-posedness of this formulation remains unclear, due to a lack of the box constraint. In addition, even assuming the box constraint, the $L^2(\Omega)$ penalty induces only very weak compactness and greatly complicates the mathematical analysis: the existence of a minimizer is only ensured in the sense of $H$-convergence and the minimizer might be matrix-valued \cite{DeckelnickHinze,LiuJinLu:2022}. Mitusch et al \cite{mitusch2021hybrid} suggested including an $H^1(\Omega)$ penalty to stabilize the training process.
Note that one should not apply the projection $P_\mathcal{A}$ in the penalty term, in order to preserve the differentiability of the objective.
\end{remark}

\subsection{Error analysis}

Now we derive (weighted) $L^2(\Omega)$ error estimates of the approximation $\P01(q^*_\theta)$. Under Assumption \ref{assum: ellip}, the solution $u^\dag\equiv u(q^\dag)$ to \eqref{equ:elliptic problem} satisfies $u^\dag\in H^2(\Omega)\cap H_0^1(\Omega)\cap W^{1,\infty}(\Omega)$ \cite[Lemma 2.1]{li2017maximal}.

\begin{assumption}\label{assum: ellip}
$f\in L^{\infty}(\Omega)$, and $q^\dag\in W^{2,p}(\Omega)\cap \mathcal{A}$ for some $p\ge \max(2,d+\mu)$ with $\mu>0$.
\end{assumption}

The next lemma gives the existence of an approximant in the admissible set $\mathfrak{P}_{p,\epsilon}$.
\begin{lemma}\label{lem:uq-uh}
Let Assumption \ref{assum: ellip} hold. Then for any $\epsilon>0$, there exists $\theta_\epsilon\in\mathfrak{P}_{p,\epsilon}$ such that
\begin{equation*}
	\|u^\dag-u_h\big(\P01(q_{\theta_\epsilon})\big) \|_{L^2(\Omega)}\leq c(h^2+\epsilon).
\end{equation*}
\end{lemma}
\begin{proof}
By the choice of $p$, $W^{1,p}(\Omega)$ continuously embeds into  $L^\infty(\Omega)$ \cite[Theorem 4.12, p. 85]{AdamsFournier:2003}.
Since $q^\dag \in W^{2,p}(\Omega)$, by Lemma \ref{lem:tanh-approx}, there exists  $\theta_\epsilon\in \mathfrak{P}_{p,\epsilon}$ such that its NN realization $q_{\theta_\epsilon}$ satisfies
\begin{equation}\label{eqn:qeps-W1p}
\|q^\dagger-q_{\theta_\epsilon}\|_{H^{1}(\Omega)}+\|q^\dagger-q_{\theta_\epsilon}\|_{L^{\infty}(\Omega)} \le c\|q^\dagger-q_{\theta_\epsilon}\|_{W^{1,p}(\Omega)} \leq c\epsilon.
\end{equation}
Then by the stability estimate \eqref{eqn:P01-approx} of the operator $\P01$, we deduce
\begin{equation}\label{eqn:qeps-W1p-1}
 \|q^\dagger-\P01(q_{\theta_\epsilon})\|_{L^{\infty}(\Omega)} \leq c\epsilon.
\end{equation}
Next we bound $\varrho_h:= u_h(\P01(q_{\theta_\epsilon}))-u_h(q^\dag)\in X_h$. It follows from the weak formulations of $u_h(\P01(q_{\theta_\epsilon}))$ and $u_h(q^\dag)$, cf. \eqref{equ:dis-vari-ellip}, and H\"{o}lder's inequality that for any $\varphi_h\in X_h$,
\begin{align*}
(\P01(q_{\theta_\epsilon})\nabla \varrho_h,\nabla\varphi_h) &= \big((q^\dag-\P01(q_{\theta_\epsilon}))\nabla u_h(q^\dag),\nabla \varphi_h \big)
\leq\|q^\dagger-\P01(q_{\theta_\epsilon})\|_{L^\infty(\Omega)}\|\nabla u_h(q^\dagger)\|_{L^2(\Omega)}\|\nabla \varphi_h\|_{L^2(\Omega)}.
\end{align*}
Next we set $\varphi_h=\varrho_h$ in the inequality. Upon noting $\P01(q_{\theta_\epsilon}) \in \mathcal{A}$, by the approximation property  \eqref{eqn:qeps-W1p-1}, Poinc{a}r\'e inequality, H\"{o}lder's inequality and the estimate $\|\nabla u_h(q^\dagger)\|_{L^2(\Omega)}\leq c \| f \|_{L^2(\Omega)}$. we obtain
\begin{equation*}
    \|w_h\|_{L^2(\Omega)}\leq c\|\nabla w_h\|_{L^2(\Omega)}\leq c\|q^\dagger-\P01(q_{\theta_\epsilon})\|_{L^\infty(\Omega)}\|\nabla u_h(q^\dagger)\|_{L^2(\Omega)}\leq c\epsilon.
\end{equation*}
This and the standard \textsl{a priori} error estimate
$\|u^\dag-u_h(q^\dag)\|_{L^2(\Omega)}\leq ch^2$
yield the desired estimate.
\end{proof}

The next lemma gives crucial \textit{a priori} bounds on $\|u^\dag-u_h(\P01(q^*_\theta))\|_{L^2(\Omega)}$ and $\| \nabla \P01(q_\theta^*)\|_{L^2(\Omega)}$.
\begin{lemma}\label{lem:pri grad qN*}
Let Assumption \ref{assum: ellip} hold. For any $\epsilon>0$, let $\theta^*\in\mathfrak{P}_{p,\epsilon}$ be a minimizer to problem \eqref{equ:dis-min-ellip}-\eqref{equ:dis-vari-ellip}. Then the following estimate holds
\begin{equation*}
\|u^\dag-u_h(\P01(q^*_\theta))\|^2_{L^2(\Omega)}+\gamma\|  \nabla \P01(q_\theta^*)\|^2_{L^2(\Omega)} \leq c(h^4+\epsilon^2+\delta^2+\gamma).
\end{equation*}
\end{lemma}
\begin{proof}
Let $q_{\theta_\epsilon}$ be the NN realization of $\theta_\epsilon\in \mathfrak{P}_{p,\epsilon}$ satisfying the estimate \eqref{eqn:qeps-W1p}, which also implies $\|q_{\theta_\epsilon}\|_{H^1(\Omega)}\le c$. Then Lemma \ref{lem:uq-uh} and the minimizing property of $q_\theta^*$, i.e., $
J_{\gamma,h}(q^*_\theta)\leq  J_{\gamma,h}(q_{\theta_\epsilon}),$
yield
\begin{align*}
&\|u_h(\P01(q^*_\theta))-z^\delta\|^2_{L^2\II}+\gamma\|\nabla q_\theta^*\|^2_{L^2\II}\leq\|u_h(\P01(q_{\theta_\epsilon}))-z^\delta\|^2_{L^2\II}+\gamma\|\nabla q_{\theta_\epsilon}\|_{L^2\II}^2 \\ \leq  & c\big(\|u_h(\P01(q_{\theta_\epsilon}))- u^\dagger\|^2_{L^2(\Omega)}+\|u^\dagger-z^\delta\|^2_{L^2(\Omega)}+\gamma\big)\leq c(h^4+\epsilon^2+\delta^2+\gamma).
\end{align*}
Applying the triangle inequality leads to
\begin{align*}
\|u^\dag-&u_h(\P01(q^*_\theta))\|^2_{L^2(\Omega)}+\gamma\|\nabla q_\theta^*\|_{L^2(\Omega)}^2
\leq c\|u^\dag-z^\delta\|^2_{L^2(\Omega)}\\
&+c\|z^\delta-u_h(\P01(q^*_\theta))\|^2_{L^2(\Omega)}+\gamma\|\nabla q_{\theta}^*\|_{L^2(\Omega)}^2
\leq c(h^4+\epsilon^2+\delta^2+\gamma).
\end{align*}
Finally, the bound on $\| \nabla \P01(q_\theta^*) \|_{L^2(\Omega)}$ follows from  \eqref{eqn:P01-stab} and the constraint $\P01(q_\theta^*)\in\mathcal{A}$.
\end{proof}

To derive an \textsl{a priori} estimate for $\P01(q_\theta^*)$, we use the following positivity condition: for some $\beta\geq0$,
\begin{equation}\label{Cond: P}
	q^\dag |\nabla u^\dag|^2+fu^\dag\geq c\ {\rm dist}(x,\partial\Omega)^\beta.
\end{equation}
Bonito et al proved that  condition \eqref{Cond: P} holds with $\beta=2$ if $\Omega$ is a Lipschitz domain, $q^\dagger\in\mathcal{A}$ and $f\in  L^2(\Omega)$ with $f>c_f$ for some $c_f>0$ \cite[Lemma 3.7]{bonito2017diffusion}, and with $\beta=0$ if $q^\dagger\in C^{1,\alpha}(\overline\Omega )\cap\mathcal{A}$, and $f\in C^{0,\alpha}(\overline\Omega)$ and $f\geq c_f>0$ on a $C^{2,\alpha}$ domain $\Omega$ for some $\alpha>0$ \cite[Lemma 3.3]{bonito2017diffusion}.

\begin{theorem}\label{thm:error-ellip}
Let Assumption \ref{assum: ellip}  hold. For any $\epsilon>0$, let $\theta^*\in\mathfrak{P}_{p,\epsilon}$ be a minimizer to problem \eqref{equ:dis-min-ellip}-\eqref{equ:dis-vari-ellip}, with $q^*_\theta$ its NN realization. Then with $\eta^2:=h^4+\epsilon^2+\delta^2+\gamma$, there holds
\begin{equation*}
\int_{\Omega}\Big(\frac{q^\dag-\P01(q_\theta^*)}{q^\dag}\Big)^2\big(q^\dag |\nabla u^\dag|^2+fu^\dag\big)\ \mathrm{d}x\leq c(\min(h^{-1}\eta+h,1)+h\gamma^{-\frac12}\eta)\gamma^{-\frac12}\eta.
\end{equation*}
Moreover, if condition \eqref{Cond: P} holds, then
\begin{equation*}
\|q^\dagger-\P01(q_\theta^*)\|_{L^2(\Omega)}\leq c\big[(\min(h^{-1}\eta+h,1)+h\gamma^{-\frac12}\eta)\gamma^{-\frac12}\eta\big]^\frac{1}{2(1+\beta)}.
\end{equation*}
\end{theorem}
\begin{proof}
By the weak formulations of $u^\dag$ and $u_h(P_\mathcal{A}(q_\theta^*))$, cf. \eqref{equ:vari problem in ellip} and \eqref{equ:dis-vari-ellip}, for any $\varphi\in H_0^1(\Omega)$, we have
\begin{align*}
&\quad\big((q^\dag-\P01(q_\theta^*))\nabla u^\dagger,\nabla\varphi\big)
= \big ((q^\dag-\P01(q_\theta^*))\nabla u^\dagger,\nabla(\varphi-P_h\varphi)\big) + \big( (q^\dag-\P01(q_\theta^*))\nabla u^\dagger,\nabla P_h\varphi\big) \\
& = -\big(\nabla\cdot((q^\dag-\P01(q_\theta^*))\nabla u^\dagger),\varphi-P_h\varphi\big)+\big(\P01(q_\theta^*)\nabla(u_h(\P01(q_\theta^*))-u^\dagger),\nabla P_h\varphi\big)=: {\rm I} + {\rm II}.
\end{align*}
Let $\varphi\equiv\frac{q^\dag-\P01(q_\theta^*)}{q^\dag}u^\dagger$. Next we bound the terms ${\rm I}$ and ${\rm II}$ separately. Direct computation gives
\begin{equation*}
\nabla\varphi=(q^\dag)^{-1}(u^\dag\nabla(q^\dag-\P01(q^*_\theta))+(q^\dag-\P01(q_\theta^*))\nabla u^\dag)-(q^\dag)^{-2}(q^\dag-\P01(q_\theta^*))(\nabla q^\dag) u^\dag.
\end{equation*}
This identity and Assumption \ref{assum: ellip} imply $\varphi\in H_0^1(\Omega)$, and further
\begin{equation}\label{eqn:bound-dphi}
\|\nabla\varphi\|_{L^2(\Omega)}\leq c(1+\|\nabla\P01(q_\theta^*)\|_{L^2(\Omega)}).
\end{equation}
Using Assumption \ref{assum: ellip} again and Lemma \ref{lem:pri grad qN*}, we obtain
\begin{align*}
\|\nabla\cdot\big((q^\dag-&\P01(q_\theta^*))\nabla u^\dagger\big)\|_{L^2(\Omega)}
\leq\|q^\dag - \P01(q_\theta^*)\|_{L^\infty(\Omega)}\|\Delta u^\dagger\|_{L^2(\Omega)} \\
&+\|\nabla q^\dag - \nabla \P01(q_\theta^*)\|_{L^2(\Omega)}\|\nabla u^\dagger\|_{L^\infty(\Omega)}
\leq c(1+\|\nabla \P01(q_\theta^*)\|_{L^2(\Omega)})\leq c\gamma^{-\frac12}\eta.
\end{align*}
Hence, we can bound the term ${\rm I}$ by
\begin{equation*}
|{\rm I}|\leq ch(1+\|\nabla\P01(q_\theta^*)\|_{L^2(\Omega)})\|\nabla \varphi\|_{L^2(\Omega)}\leq ch(1+\|\nabla \P01(q_\theta^*)\|^2_{L^2(\Omega)})\leq ch\gamma^{-1}\eta^2.
\end{equation*}
By the Cauchy--Schwarz inequality and the estimate \eqref{eqn:bound-dphi}, we can bound the term ${\rm II}$ by
\begin{align*}
|{\rm II}|
&\leq \|\P01(q_\theta^*)\|_{L^{\infty}(\Omega)}\|\nabla(u_h(\P01(q_\theta^*))-u^\dagger)\|_{L^{2}(\Omega)}\|\nabla\varphi\|_{L^2(\Omega)}\\
&\leq c \|\nabla(u_h(\P01(q_\theta^*))-u^\dagger)\|_{L^{2}(\Omega)}\|\nabla\varphi\|_{L^2(\Omega)}\\
&\leq c\big(1+\|\nabla \P01(q_\theta^*)\|_{L^2(\Omega)}\big)\|\nabla(u_h(\P01(q_\theta^*))-u^\dagger)\|_{L^{2}\II}.
\end{align*}
Then by Lemma \ref{lem:pri grad qN*}, the inverse inequality in the space $X_h$ \cite[(1.12), p. 4]{Thomee:2006}, the approximation property \eqref{inequ: L2 proj approx} and the $L^2(\Omega)$-stability of $P_h$ and the regularity $u^\dag \in H^2(\Omega)$, we can bound the term $\rm II$ by
\begin{align*}
|{\rm II}|
&\le  c \gamma^{-\frac12}\eta\,
\big(\| \nabla(u_h(\P01(q_\theta^*))- P_h u^\dagger)  \|_{L^2(\Omega)}
+ \| \nabla(P_h u^\dagger-u^\dagger) \|_{L^2\II} \big) \\
&\le c \gamma^{-\frac12}\eta\,
\big( h^{-1} \|  u_h(\P01(q_\theta^*))  - P_h u^\dagger \|_{L^2(\Omega)}
+ h \| u^\dagger \|_{H^2\II} \big)\\
&\le c \gamma^{-\frac12}\eta\,
\big( h^{-1} \|  u_h(\P01(q_\theta^*))- u^\dagger \|_{L^2(\Omega)}
+ h \| u^\dagger \|_{H^2\II} \big) \le c  \gamma^{-\frac12}\eta\, (h^{-1}\eta +h).
\end{align*}
Further, the estimate $
\|\nabla u_h(P_{\mathcal{A}}(q_\theta^*))\|_{L^2(\Omega)}\leq c \| f \|_{L^2(\Omega)}$ and the regularity $u^\dag \in H^2(\Omega)$ imply
$\|\nabla(u_h(\P01(q_\theta^*))-u^\dagger)\|_{L^{2}\II}\leq c$. Hence,
$|{\rm II}| \le c \gamma^{-\frac12} \eta.$ Combining these estimates on ${\rm II}$ yields
\begin{align*}
|{\rm II}|
 \le c  \gamma^{-\frac12}\eta\, \min(h^{-1}\eta +h,1).
\end{align*}
Moreover, direct computation gives \cite[Theorem 2.2]{bonito2017diffusion}
\begin{equation*}
 \big((q^\dag-\P01(q_\theta^*))\nabla u^\dag,\nabla\varphi\big) = \frac12\int_{\Omega}\Big(\frac{q^\dag-\P01(q_\theta^*)}{q^\dag}\Big)^2\big(q^\dag |\nabla u^\dag|^2+fu^\dag\big)\ \mathrm{d}x.
\end{equation*}
This and the preceding bounds together show the first assertion. To bound $\|q^\dagger-\P01(q_\theta^*)\|_{L^2(\Omega)}$, we fix $\rho>0$, and divide $\Omega$ into two disjoint sets $\Omega=\Omega_\rho\cup \Omega_\rho^c$, with $\Omega_\rho=\{x\in\Omega:{\rm dist}(x,\partial\Omega)\geq\rho\}$ and $\Omega_\rho^c=\Omega\setminus\Omega_\rho$. Then by condition \eqref{Cond: P} and the box constraint $q^\dag\in \mathcal{A}$, we have
\begin{align*}
&\|q^\dagger-\P01(q_\theta^*)\|_{L^2(\Omega_\rho)}^2  = \rho^{-\beta}\int_{\Omega_\rho }(q^\dagger-\P01(q_\theta^*))^2\rho^{\beta}{\rm d}x
\leq \rho^{-\beta}\int_{\Omega_\rho }(q^\dag-\P01(q_\theta^*))^2\mathrm{dist}(x,\partial\Omega)^{\beta}{\rm d}x\\
\leq&c\rho^{-\beta}\int_{\Omega_\rho }\Big(\frac{q^\dag-\P01(q_\theta^*)}{q^\dag}\Big)^2\big(q^\dag |\nabla u^\dagger|^2+fu^\dagger\big){\rm d}x
\leq c\rho^{-\beta}(\min(h^{-1}\eta+h,1)+h\gamma^{-\frac12}\eta)\gamma^{-\frac12}\eta.
\end{align*}
Meanwhile, on the subdomain $\Omega_\rho^c$, the box constraint $P_\mathcal{A}(q_\theta^*)\in \mathcal{A}$ leads to
\begin{equation*}
\|q^\dagger-\P01(q_\theta^*)\|_{L^2(\Omega^c_\rho)}^2\leq  c|\Omega_\rho^c| \|q^\dagger-\P01(q_\theta^*)\|^{2}_{L^{\infty}(\Omega_\rho^c)}\leq c\rho.
\end{equation*}
Combining the last two estimates and then optimizing over $\rho>0$
complete the proof.
\end{proof}

\begin{remark}
Theorem \ref{thm:error-ellip} provides useful guidelines for choosing the algorithmic parameters: $\gamma=\mathcal{O}(\delta^2)$, $h=\mathcal{O}(\delta^\frac12)$ and $\epsilon=\mathcal{O}(\delta)$. Then under condition \eqref{Cond: P}, we obtain $\|q^\dag-P_\mathcal{A}(q_\theta^*)\|_{L^2(\Omega)}\leq c\delta^\frac{1}{4(1+\beta)}.$
This result is comparable with that for the purely FEM approximation \cite[Corollary 3.3]{jin2021error}.
\end{remark}

\subsection{Quadrature error analysis}

The weak formulation and objective requires evaluating various integrals. This is commonly done via a quadrature scheme. While this issue is direct for the standard FEM \cite{ciarlet2002finite}, it is nontrivial when NNs are involved: NNs are globally supported and no longer polynomials within each finite element. Thus, the use of quadrature schemes is required, and there is an inevitable quadrature error, which may influence the accuracy of the NN approximation \cite{BerroneCanuto:2022,RiveraPardo:2022}. We aim to provide a quadrature error analysis.

There are many possible quadrature rules \cite[Chapter 15]{Thomee:2006}. We focus on one simple scheme to shed useful insights. On each element $K\in \mathcal{T}_h$,
we uniformly divide it into $2^{dn}$ sub-simplexes,  denoted by $\{ K_i\}_{i=1}^{2^{dn}}$, with the uniform diameter $h_K/2^n$. The division for $d=1,2$
is trivial, and for $d=3$, it is also feasible  \cite{Ong:1994}.
Then consider the following quadrature rule over the element $K$ (with $P_j^i$ denoting the $j$th node of the $i$th sub-simplex $K_i$):
\begin{equation*}
Q_K(v) = \sum_{i=1}^{2^{dn}} \frac{|K_i|}{d+1} \sum_{j=1}^{d+1} v(P_j^i),\quad \forall v\in C(\overline K).
\end{equation*}
The embedding $H^2\II \hookrightarrow L^\infty\II$ (for $d=1,2,3$) and Bramble--Hilbert lemma \cite[Theorem 4.1.3]{ciarlet2002finite} lead to
\begin{equation}\label{ineq: quad approx-0}
\Big|\int_{K} v \ \mathrm{d}x - Q_K(v)\Big| \leq
 c |K|^{\frac12} 2^{-2n} h_K^2  |v|_{H^2(K)}, \quad \forall v\in H^2(K), 
\end{equation}
Then we can define a global quadrature rule:
\begin{equation}\label{eqn:quad-global}
Q_h(v) =  \sum_{K\in \mathcal{T}_h} Q_K(v),\quad \forall v\in C(\overline \Omega),
\end{equation}
which satisfies the following error estimate
\begin{equation}\label{ineq: quad approx}
 \Big|\int_{\Omega} v \ \mathrm{d}x - Q_h(v)\Big| \leq c 2^{-2n} h^2 |v|_{H^{2}(\Omega)},\quad \forall v\in H^2(\Omega).
\end{equation}
Similarly, we define a discrete / broken $L^2(\Omega)$ inner product $(\cdot,\cdot)_h$ by
\begin{equation*}
(w,v)_h := Q_h(wv) = \sum_{K\in\mathcal{T}_h}Q_K(wv),\quad \forall w,v\in C(\overline \Omega).
\end{equation*}

\begin{lemma}\label{lem:quad-error}
The following error estimate holds for any $v_h,w_h\in X_h$, and $n\in\mathbb{N}$:
\begin{align*}|(q\nabla v_h,\nabla w_h)-(q \nabla v_h,\nabla w_h)_h|\leq c (2^{-n}h)^p \|q\|_{W^{p,\infty}(\Omega)}\|\nabla v_h\|_{L^2(\Omega)}\|\nabla w_h\|_{L^2(\Omega)},\quad \mbox{with } p=1,2;
\end{align*}
\end{lemma}
\begin{proof}
Let $\Pi_{K_j}: C(K_j) \rightarrow P_1(K_j)$ be the Lagrange nodal interpolation operator on the sub-simplex $K_j$. Since the quadrature rule on $K_j$ is exact for $P_1(K_j)$, we have
$$(q\nabla v_h,\nabla w_h)-(q \nabla v_h,\nabla w_h)_h =  \sum_{K\in \mathcal{T}_h} \sum_{j=1}^{2^{dn}}
\int_{K_j} (q - \Pi_{K_j}q) \nabla v_h\cdot\nabla w_h\,\d x$$
Then the local estimate for Lagrange interpolation leads to
\begin{equation*}
\begin{aligned}
&|(q\nabla v_h,\nabla w_h)-(q \nabla v_h,\nabla w_h)_h|
\le \sum_{K\in \mathcal{T}_h} \sum_{j=1}^{2^{dn}}
\int_{K_j} \Big| (q - \Pi_{K_j}q) \nabla v_h\cdot\nabla w_h \Big| \,\d x\\
 \le& c \sum_{K\in \mathcal{T}_h} \sum_{j=1}^{2^{dn}} (2^{-n} h)^p \| q \|_{W^{p,\infty}(K_j)}
 \| \nabla v_h\|_{L^2(K_j)} \|\nabla w_h \|_{L^2(K_j)}
 \le c  (2^{-n} h)^p \| q \|_{W^{p,\infty}(\Omega)}
 \| \nabla v_h\|_{L^2\II} \|\nabla w_h \|_{L^2\II}.
\end{aligned}
\end{equation*}
This proves the desired estimate.
\end{proof}

Then the hybrid NN-FEM approximation of problem \eqref{equ:Tikhonov problem in ellip}-\eqref{equ:vari problem in ellip} (with numerical integration) reads
\begin{equation}\label{equ:dis-min-ellip-q}
	\min_{\theta\in\Theta(R)} \tilde{J}_{\gamma,h}(q_\theta)=\frac{1}{2}\|\tilde u_h( \P01(q_\theta))-z^{\delta}\|^2_{L^2(\Omega)}+\frac{\gamma}{2} Q_h(|\nabla q_\theta|^2),
\end{equation}
where $\tilde u_h\equiv \tilde u_h(\P01(q_\theta))\in X_h$ satisfies the following discrete variational problem
\begin{equation}\label{equ:dis-vari-ellip-q}	
( \P01(q_\theta)\nabla \tilde {u}_h,\nabla\varphi_h)_h=(f,\varphi_h), \quad\forall \varphi_h\in X_h.
\end{equation}
We focus on approximating the integrals involving NNs only.
The variational problem \eqref{equ:dis-vari-ellip-q} involves also the quadrature approximation, which necessitates quantifying the associated error.
The presence of $P_\mathcal{A}$ in the weak formulation ensures the $X_h$-ellipticity of the broken $L^2(\Omega)$ semi-inner product, and hence the unique existence of the discrete forward map is ensured \cite{ciarlet1991basic,abdulle2012priori}.
Then repeating the argument for problem \eqref{equ:dis-min-ellip}-\eqref{equ:dis-vari-ellip} yields the well-posedness of problem \eqref{equ:dis-min-ellip-q}-\eqref{equ:dis-vari-ellip-q}. The analysis of the quadrature error requires the following condition on the problem data.
\begin{assumption}\label{assum: ellip quad}
$q^\dag\in W^{2,\infty}(\Omega)\cap \mathcal{A}$ and $f\in L^{\infty}(\Omega)$.
\end{assumption}

Next we state an analogue of Lemma \ref{lem:uq-uh} in the presence of numerical integration.
\begin{lemma}\label{lem:uq-uh:quad}
Let Assumption \ref{assum: ellip quad} hold. Then for any $\epsilon>0$, there exists  $\theta_\epsilon\in \mathfrak{P}_{\infty,\epsilon}$ such that
 \begin{equation*}
\|u^\dag- \tilde {u}_h(  \P01({q}_{\theta_\epsilon}))\|_{L^2(\Omega)}\leq c(h^2+\epsilon).
\end{equation*}
\end{lemma}
\begin{proof}
The proof is similar to Lemma \ref{lem:uq-uh}. First, under Assumption \ref{assum: ellip quad}, there holds $\|u^\dag-\tilde{u}_h(  q^\dag)\|_{L^2(\Omega)}\leq ch^2$ \cite[Theorem 5]{abdulle2012priori}.
Next by Lemma  \ref{lem:tanh-approx}, there exists $\theta_{\epsilon} \in \mathfrak{P}_{\infty,\epsilon}$ such that its NN realization $q_{\theta_\epsilon}$ satisfies
\begin{equation}\label{eqn:qeps-W1p-2}
\|q^\dag- {q}_{\theta_\epsilon}\|_{W^{1,\infty}(\Omega)}\leq \epsilon.
\end{equation}
Then by the stability estimate  \eqref{eqn:P01-approx} of the operator $\P01$, we deduce
\begin{equation}\label{eqn:qeps-W1inf}
 \|q^\dag-\P01(q_{\theta_\epsilon})\|_{L^{\infty}(\Omega)} \leq \epsilon.
 \end{equation}
Let $\tilde{w}_h:= \tilde{u}_h( \P01( {q}_{\theta_\epsilon}))-\tilde{u}_h(q^\dag)\in X_h$. Repeating the argument of Lemma \ref{lem:uq-uh} yields
\begin{align*}
c_0\| \nabla \tilde{w}_h\|^2_{L^2\II}&\leq ( \P01( {q}_{\theta_\epsilon})\nabla \tilde{w}_h,\nabla \tilde{w}_h)_h
= \big( (q^\dag-\P01( {q}_{\theta_\epsilon}))\nabla \tilde{u}_h(q^\dag),\nabla \tilde{w}_h\big)_h\\
 &\leq\| q^\dag-\P01( {q}_{\theta_\epsilon})\|_{L^\infty(\Omega)}\|\nabla \tilde{u}_h(q^\dag)\|_{L^2\II} \|\nabla \tilde{w}_h\|_{L^2\II},
\end{align*}
since
$|(q\nabla u_h,\nabla v_h)_h|\leq c\|q\|_{L^\infty(\Omega)}\|\nabla u_h\|_{L^2(\Omega)}\|\nabla v_h\|_{L^2(\Omega)}$.
Using the estimate \eqref{eqn:qeps-W1inf}, the stability of Lagrange interpolation and the \textit{a priori} estimate $\|\nabla \tilde{u}_h( q^\dagger)\|_{L^2\II}\le c$ and Poincar\'{e} inequality, we  deduce
\begin{align*}
\|\tilde{w}_h\|_{L^2(\Omega)}&\leq c\|\nabla \tilde{w}_h\|_{L^2(\Omega)}\leq c\| q^\dag-\P01(q_{\theta_\epsilon})\|_{L^\infty(\Omega)}\|\nabla \tilde{u}_h( q^\dag)\|_{L^2(\Omega)}
\leq c\| q^\dag-\P01(q_{\theta_\epsilon})\|_{L^\infty(\Omega)} \le c\epsilon.
\end{align*}
This completes the proof of the lemma.
\end{proof}

The next lemma provides an \textsl{a priori} bound on $\|u^\dag-\tilde{u}_h(\P01(q^*_\theta))\|_{L^2(\Omega)}$ and $\nabla q^*_\theta$.

\begin{lemma}\label{lem:pri grad qN* quad}
Let Assumption \ref{assum: ellip quad} hold. Fix $\epsilon>0$, and let $\theta^*\in\mathfrak{P}_{\infty,\epsilon} $ be a minimizer to problem \eqref{equ:dis-min-ellip-q}-\eqref{equ:dis-vari-ellip-q} and $q^*_\theta$ its NN realization. Then the following estimate holds
\begin{equation*}
\|u^\dag -\tilde{u}_h( \P01 (q^*_\theta))\|^2_{L^2(\Omega)}+\gamma Q_h(|\nabla  q^*_\theta|^2)\leq c(h^4+\epsilon^2+\delta^2+\gamma).
\end{equation*}
\end{lemma}
 \begin{proof}
The proof relies on the minimizing property of $\theta^*$, Lemma \ref{lem:uq-uh:quad} and the existence of an element $ q_{\theta_\epsilon}\in W^{1,\infty}(\Omega)$ satisfying \eqref{eqn:qeps-W1p-2}. The estimate \eqref{eqn:qeps-W1p-2} and the regularity $q^\dag\in W^{2,\infty}(\Omega)$ implies $\|q_\epsilon\|_{W^{1,\infty}(\Omega)}\leq c$. This  yields $Q_h(|\nabla q_\epsilon|^2)\leq c$, since the quadrature operator $Q_h$ is stable on $C(\overline{\Omega})$. The rest of the proof is identical with that for Lemma \ref{lem:pri grad qN*}, and hence, we omit the details.
\end{proof}

Next we show an \textit{a priori} bound on quadrature error of the penalty term.
\begin{lemma}\label{lem:quad-err-H1}
Let $\theta \in \mathfrak{P}_{\infty,\epsilon}$, of depth $L$, width $W$ and bound $R$, and $v_\theta$ be its NN realization, with $RW>2$. Then the following quadrature error estimate holds
\begin{align*}
\|\nabla v_\theta\|^2_{L^2(\Omega)}-Q_h(\|\nabla v_\theta\|_{\ell^2}^2)\leq c2^{-2n} h^2 |\sum_{i=1}^d(\partial_{x_i} v_\theta)^2|_{W^{2,\infty}(\Omega)}\leq c2^{-2n}h^2R^{4L}{W}^{4L-4}.
\end{align*}
\end{lemma}
\begin{proof}
By the NN realization \eqref{equ: network realization}, we have for every layer $\ell\in [L-1]$ and each $i\in [d_\ell]$,
$v^{(\ell)}_i=\rho\Big(\sum_{j=1}^{d_{\ell-1}}A^{(\ell)}_{ij}v^{(\ell-1)}_j+b^{(\ell)}_i\Big)$.
Then for all $1\leq k,m\leq d$, direct computation with the chain rule gives
\begin{align*}
	\partial^2_{ x_k,x_m}v^{(\ell)}_i=&\rho''\Big(\sum_{j=1}^{d_{\ell-1}}A^{(\ell)}_{ij}v^{(\ell-1)}_j+b^{(\ell)}_i\Big)\Big(\sum_{j=1}^{d_{\ell-1}}A^{(\ell)}_{ij}\partial_{ x_k}v^{(\ell-1)}_j\Big)\Big(\sum_{j=1}^{d_{\ell-1}}A^{(\ell)}_{ij}\partial_{ x_m}v^{(\ell-1)}_j\Big)\\&+ \rho'\Big(\sum_{j=1}^{d_{\ell-1}}A^{(\ell)}_{ij}v^{(\ell-1)}_j+b^{(\ell)}_i\Big)\Big(\sum_{j=1}^{d_{\ell-1}}A^{(\ell)}_{ij}\partial^2_{x_k,x_m}v^{(\ell-1)}_j\Big).
\end{align*}
Note that for the tanh activation function $\rho$, $\|\rho'\|_{L^\infty(\mathbb{R})}\leq 1,\|\rho''\|_{L^\infty(\mathbb{R})}\leq 1$, cf. Lemma \ref{lem:rho}, and further \cite[Lemma 3.4, eq. (3.6)]{jin2022imaging}
\begin{equation}\label{eqn:dvthe}
\|\partial_{x_k}v_j^{(\ell)}\|_{L^\infty(\Omega)}\leq R^{\ell}{W}^{\ell-1},\quad
\forall \ell\in [L-1], j\in [d_\ell].
\end{equation}
Then we arrive at
\begin{align*}
	\|\partial^2_{ x_k,x_m}v^{(\ell)}_i\|_{L^\infty(\Omega)}&\leq R^2{W}^2\max_{j\in[d_{d_{\ell-1}}]}\|\partial_{x_k} v_j^{(\ell-1)}\|_{L^\infty(\Omega)}\max_{j\in[d_{d_{\ell-1}}]}\|\partial_{x_m} v_j^{(\ell-1)}\|_{L^\infty(\Omega)}\\
   &\quad +R{W}\max_{j\in[d_{d_{\ell-1}}]}\|\partial^2_{ x_k,x_m}v_j^{(\ell-1)}\|_{L^\infty(\Omega)}\leq R^{2\ell}{W}^{2\ell-2}+R{W}\max_{j\in[d_{\ell-1}]}\|\partial^2_{ x_k,x_m}v_j^{(\ell-1)}\|_{L^\infty(\Omega)}.
\end{align*}
Note also the trivial estimate \begin{align*}
\|\partial^2_{ x_k,x_m}v^{(1)}_i\|_{L^\infty(\Omega)}\leq \bigg\|\rho''\Big(\sum_{j=1}^{d}A^{(1)}_{ij}x_j+b^{(1)}_i\Big)A^{(1)}_{ik}A^{(1)}_{im}\bigg\|_{L^\infty(\Omega)}\leq R^2.
\end{align*}
Taking maximum in $i\in[d_\ell]$ and then applying the inequality recursively lead to
\begin{align*}
\max_{i\in[d_\ell]}\|\partial_{x_k,x_m}v_i^{(\ell)}\|_{L^\infty(\Omega)} & \leq R^2\sum_{j=1}^{\ell-1} (RW)^{2j} (RW)^{\ell-1-j} + (RW)^{\ell-1}\max_{i\in [d_1]}\|\partial_{x_k,x_m}v_i^{(1)}\|_{L^\infty(\Omega)}\\
&= R^{2\ell}W^{2\ell-2} \sum_{j=0}^{\ell-1}(RW)^{-j} \leq \frac{R^{2\ell}W^{2\ell-2}}{1-(RW)^{-1}}.
\end{align*}
Hence, under the condition $RW\geq2$, we may bound
\begin{equation}\label{ineq: 2nd diff bound}
\|\partial^2_{ x_k,x_s}v^{(\ell)}_i\|_{L^\infty(\Omega)}\leq 2R^{2\ell}W^{2\ell-2},\quad \forall \ell\in\in [L], i\in [d_\ell].
\end{equation}
By direct computation, we obtain for $1\leq k,m,n\leq d$
\begin{align*}
	\partial^3_{ x_k,x_m,x_n}v^{(\ell)}_i=&\rho'''\Big(\sum_{j=1}^{d_{\ell-1}}A^{(\ell)}_{ij}v^{(\ell-1)}_j+b^{(\ell)}_i\Big)\Big(\sum_{j=1}^{d_{\ell-1}}A^{(\ell)}_{ij}\partial_{ x_k}v^{(\ell-1)}_j\Big)\Big(\sum_{j=1}^{d_{\ell-1}}A^{(\ell)}_{ij}\partial_{ x_m}v^{(\ell-1)}_j\Big)\Big(\sum_{j=1}^{d_{\ell-1}}A^{(\ell)}_{ij}\partial_{ x_n}v^{(\ell-1)}_j\Big)\\&+\rho''\Big(\sum_{j=1}^{d_{\ell-1}}A^{(\ell)}_{ij}v^{(\ell-1)}_j+b^{(\ell)}_i\Big)\Big(\sum_{j=1}^{d_{\ell-1}}A^{(\ell)}_{ij}\partial^2_{x_k,x_n}v^{(\ell-1)}_j\Big)\Big(\sum_{j=1}^{d_{\ell-1}}A^{(\ell)}_{ij}\partial_{ x_m}v^{(\ell-1)}_j\Big)\\&+\rho''\Big(\sum_{j=1}^{d_{\ell-1}}A^{(\ell)}_{ij}v^{(\ell-1)}_j+b^{(\ell)}_i\Big)\Big(\sum_{j=1}^{d_{\ell-1}}A^{(\ell)}_{ij}\partial^2_{x_m,x_n}v^{(\ell-1)}_j\Big)\Big(\sum_{j=1}^{d_{\ell-1}}A^{(\ell)}_{ij}\partial_{ x_k}v^{(\ell-1)}_j\Big)\\&+ \rho''\Big(\sum_{j=1}^{d_{\ell-1}}A^{(\ell)}_{ij}v^{(\ell-1)}_j+b^{(\ell)}_i\Big)\Big(\sum_{j=1}^{d_{\ell-1}}A^{(\ell)}_{ij}\partial^2_{x_k,x_m}v^{(\ell-1)}_j\Big)\Big(\sum_{j=1}^{d_{\ell-1}}A^{(\ell)}_{ij}\partial_{ x_n}v^{(\ell-1)}_j\Big)\\&+\rho'\Big(\sum_{j=1}^{d_{\ell-1}}A^{(\ell)}_{ij}v^{(\ell-1)}_j+b^{(\ell)}_i\Big)\Big(\sum_{j=1}^{d_{\ell-1}}A^{(\ell)}_{ij}\partial^3_{x_k,x_m,x_n}v^{(\ell-1)}_j\Big).
\end{align*}
This, along with the bound $\|\rho'''\|_{L^\infty(\mathbb{R})}\leq 2$ from Lemma \ref{lem:rho}, implies
\begin{align*}
 \|\partial^3_{ x_k,x_m,x_n}v^{(\ell)}_i&\|_{L^\infty(\Omega)}\leq RW\max_{j\in[d_{\ell-1}]}\|\partial^3_{ x_k,x_m,x_n}v_j^{(\ell-1)}\|_{L^\infty(\Omega)}\\
 &+ 2 R^3W^3\max_{j\in[d_{\ell-1}]}\|\partial_{x_k} v_j^{(\ell-1)}\|_{L^\infty(\Omega)}\max_{j\in[d_{\ell-1}]}\|\partial_{x_m} v_j^{(\ell-1)}\|_{L^\infty(\Omega)}\max_{j\in[d_{\ell-1}]}\|\partial_{x_n} v_j^{(\ell-1)}\|_{L^\infty(\Omega)} \\
 &+ R^2W^2\Big(\max_{j\in[d_{\ell-1}]}\|\partial^2_{ x_k,x_n}v_j^{(\ell-1)}\|_{L^\infty(\Omega)}\max_{j\in[d_{\ell-1}]}\|\partial_{ x_m}v_j^{(\ell-1)}\|_{L^\infty(\Omega)}\\
  &+\max_{j\in[d_{\ell-1}]}\|\partial^2_{ x_m,x_n}v_j^{(\ell-1)}\|_{L^\infty(\Omega)}\max_{j\in[d_{\ell-1}]}\|\partial_{ x_k}v_j^{(\ell-1)}\|_{L^\infty(\Omega)}\\
  &+\max_{j\in[d_{\ell-1}]}\|\partial^2_{ x_k,x_m}v_j^{(\ell-1)}\|_{L^\infty(\Omega)}\max_{j\in[d_{\ell-1}]}\|\partial_{ x_n}v_j^{(\ell-1)}\|_{L^\infty(\Omega)}\Big).
\end{align*}
Then it follows from the estimates \eqref{eqn:dvthe} and \eqref{ineq: 2nd diff bound} and
the condition $RW\geq 2$ that
\begin{align*}
	\|\partial^3_{ x_k,x_m,x_n}v^{(\ell)}_i\|_{L^\infty(\Omega)}&\leq 2R^{3\ell}W^{3\ell-3}+6R^{3\ell-1}W^{3\ell-4}+RW\max_{j\in[d_{\ell-1}]}\|\partial^3_{ x_k,x_m,x_n}v_j^{(\ell-1)}\|_{L^\infty(\Omega)}\\
 & \leq 5R^{3\ell}W^{3\ell-3}+RW\max_{j\in[d_{\ell-1}]}\|\partial^3_{ x_k,x_m,x_n}v_j^{(\ell-1)}\|_{L^\infty(\Omega)}.
\end{align*}
Meanwhile, direct computation gives the estimate
\begin{equation*}
    \|\partial^3_{ x_k,x_m,x_n}v^{(1)}_i\|_{L^\infty(\Omega)}\leq \bigg\|\rho'''\Big(\sum_{j=1}^{d}A^{(1)}_{ij}x_j+b^{(1)}_i\Big)A^{(1)}_{ik}A^{(1)}_{im}A^{(1)}_{in}\bigg\|_{L^\infty(\Omega)}\leq 2R^3.
\end{equation*}
The last two estimates together and the condition $RW\geq2$ yield
\begin{equation}\label{ineq: 3rd diff bound}
	\|\partial^3_{ x_k,x_s,x_t}v^{(\ell)}_i\|_{L^\infty(\Omega)}\leq 10R^{3\ell}{W}^{3\ell-3},\quad \forall \ell \in [L].
\end{equation}
Substituting this estimate into \eqref{eqn:quad-global}  completes the proof of the lemma.
\end{proof}

Now we can state the error estimate on the approximation $\tilde q_\theta^*$.
\begin{theorem}\label{thm:error-ellip-q}
Let Assumption \ref{assum: ellip quad} hold. Fix $\epsilon>0$, and let $\tilde{\theta}^*\in \mathfrak{P}_{\infty,\epsilon}$ be a minimizer to problem \eqref{equ:dis-min-ellip-q}-\eqref{equ:dis-vari-ellip-q} and $\tilde q_\theta^*$ be its NN realization. Then with $\eta^2:=h^4+\epsilon^2+\delta^2+\gamma$ and $\zeta:=1+\gamma^{-1}\eta^2+2^{-2n} h^2 R^{4L}{W}^{4L-4}$, we have
\begin{align*}		
\int_{\Omega}\Big(\frac{q^\dag- \P01(\tilde q_\theta^*)}{q^\dag}\Big)^2\big(q^\dag |\nabla u^\dag|^2+fu^\dag\big)\ \mathrm{d}x
\leq c\big(h\zeta^\frac12+ (\min(h^{-1}\eta+h,1) +  2^{-n}h R^{L}{W}^{L} \big)\zeta^\frac12.
\end{align*}
Moreover, if condition \eqref{Cond: P} holds, then
\begin{equation*}
\|q^\dagger-\P01(\tilde q_\theta^*)\|_{L^2(\Omega)}\leq c\big(h\zeta^\frac12+ (\min(h^{-1}\eta+h,1) +  2^{-n}h R^{L}{W}^{L} )\zeta^\frac12\big)^{\frac{1}{2(\beta+1)}}.
\end{equation*}
\end{theorem}
\begin{proof}
By the weak formulations of $u^\dag$ and $\tilde u_h(P_\mathcal{A}(\tilde q_\theta^*)$, cf. \eqref{equ:vari problem in ellip} and \eqref{equ:dis-vari-ellip-q}, we have for any  $\varphi\in H_0^1(\Omega)$,
\begin{align*}
 \big((q^\dag-  &\P01(\tilde q_\theta^*))\nabla u^\dagger,\nabla\varphi\big)
= \big ((q^\dag- \P01(\tilde q_\theta^*))\nabla u^\dagger,\nabla(\varphi-P_h\varphi)\big) + \big( (q^\dag- \P01(\tilde q_\theta^*))\nabla u^\dagger,\nabla P_h\varphi\big) \\
=& -\big(\nabla\cdot((q^\dag- \P01(\tilde q_\theta^*))\nabla u^\dagger),\varphi-P_h\varphi\big)+\big( \P01(\tilde q_\theta^*)\nabla(\tilde{u}_h(\P01(\tilde q_\theta^*))-u^\dagger),\nabla P_h\varphi\big)  \\
& + [\big( \P01(\tilde q_\theta^*)\nabla \tilde{u}_h(\P01(\tilde q_\theta^*)),\nabla P_h\varphi\big)_h - \big( \P01(\tilde q_\theta^*)\nabla \tilde{u}_h(\P01(\tilde q_\theta^*)),\nabla P_h\varphi\big)]  =: {\rm I + II + III}.
\end{align*}
Next we set $\varphi\equiv\frac{q^\dag- \P01(\tilde q_\theta^*)}{q^\dag}u^\dag$ in the identity and bound the three terms separately. By the stability estimate \eqref{eqn:P01-stab} of the operator $P_\mathcal{A}$ and Lemmas \ref{lem:pri grad qN* quad} and \ref{lem:quad-err-H1}, we have
\begin{align*}
\|\nabla  \P01(\tilde q_\theta^*)\|^2_{L^2\II}
&\le \| \nabla  \tilde q_\theta^*\|^2_{L^2\II}
=Q_h(|\nabla\tilde q_\theta^*|^2) +  [\| \nabla{\tilde q}_\theta^*\|^2_{L^2(\Omega)}-Q_h(|\nabla\tilde q_\theta^*|^2)]\\
&\le c (\gamma^{-1}\eta^2 \, + 2^{-2n}h^2R^{4L}W^{4L-4}).
\end{align*}
Thus we can bound $\|\nabla \varphi\|_{L^2(\Omega)}$ by
\begin{equation}\label{eqn:nablaphi-quad}
\|\nabla \varphi\|_{L^2(\Omega)}\leq c(1+\|\nabla  \P01(\tilde q_\theta^*)\|_{L^2(\Omega)})
\leq c\zeta^\frac12.
\end{equation}
Repeating the argument of Theorem \ref{thm:error-ellip} and applying Lemma \ref{lem:pri grad qN* quad} yield
\begin{align*}
|{\rm I}| &\leq ch(1+\|\nabla\P01{q}_\theta^*\|^2_{L^2(\Omega)})\leq ch\zeta,\\
|{\rm II}| &\le  c\big(1+\|\nabla \P01(q_\theta^*)\|_{L^2(\Omega)}\big)\|\nabla(\tilde u_h(\P01(q_\theta^*))-u^\dagger)\|_{L^{2}\II} \le c \min(h^{-1}\eta+h,1) \zeta^\frac12.
\end{align*}
Next from Lemma \ref{lem:quad-error} (with $p=1$), the stability of $P_\mathcal{A}$ and the bound \eqref{eqn:dvthe}, we deduce
\begin{align*}
|{\rm III}|&\leq c2^{-n}h\|\P01(\tilde q_\theta^*)\|_{W^{1,\infty}(\Omega)}\|\nabla\tilde u_h(\P01(\tilde q_\theta^*))\|_{L^2(\Omega)}
\|\nabla P_h \varphi\|_{L^2(\Omega)}\\
&\leq  c 2^{-n}h  \zeta^{\frac12}
(\|  \P01( \tilde q_\theta^*) \|_{L^\infty\II} + \| \nabla \P01(\tilde q_\theta^*) \|_{L^\infty\II})\\
&\leq  c 2^{-n}h \zeta^{\frac12}
(1 + \| \nabla \tilde q_\theta^* \|_{L^\infty\II})
\leq c 2^{-n}h R^LW^{L} \zeta ^\frac12.
\end{align*}
The proof of the second assertion is identical with that of Theorem \ref{thm:error-ellip}.
\end{proof}

\begin{remark}
Theorem \ref{thm:error-ellip-q} indicates that the error estimate in the presence of numerical quadrature is similar to the case of exact integration, provided that the quadrature error is sufficiently small. The quadrature error involves a factor $R^{4L}W^{4L-1}$, which can be large for deep NNs, and hence it may require a large $n$ to compensate its influence on the reconstruction $\P01(\tilde q_\theta^*)$. Indeed, one may take $2^{-2n} h^2 R^{4L}{W}^{4L-4}=\mathcal{O}(1)$. This and the choice $\tilde{\theta}^*\in \mathfrak{P}_{\infty,\epsilon} $, i.e., $L=\mathcal{O}(\log(d+2))$, $N_\theta=\mathcal{O}(\epsilon^{-\frac{d}{1-\mu}})$ and $R=\mathcal{O}(\epsilon^{-2-\frac{2+3d}{1-\mu}})$ directly imply $n=\mathcal{O}(d|\log \epsilon|)$. This estimate is a bit pessimistic. In practice, the choice $n=0$ suffices the desired accuracy.
\end{remark}

\section{Parabolic inverse problem}
\label{sec:parabolic}

In this section, we extend the approach to the parabolic case:
\begin{equation}\label{equ:parabolic problem}
	\left\{
	\begin{aligned}
		\partial_tu-\nabla\cdot(q\nabla u) &= f, \ &\mbox{in}&\ \Omega\times(0,T), \\
		u&=0, \ &\mbox{on}&\ \partial\Omega\times(0,T), \\
		u(0)&=u_0, \ &\mbox{in}&\ \Omega.
	\end{aligned}
	\right.
\end{equation}
Like before, we are given the observation $z^\delta$ on the space-time domain $\Omega\times (T_0,T)$ (with $0\leq T_0<T$):
\begin{equation*}
		z^\delta(x,t) = u(q^\dag)(x,t) + \xi(x,t), \quad (x,t)\in \Omega\times(T_0,T),
\end{equation*}
where $\xi$ denotes the measurement noise, with a noise level
$\delta=\|u(q^\dagger)-z^\delta\|_{L^2(T_0,T;L^2(\Omega))}$.
We aim at recovering the coefficient $q\in \mathcal{A}$ from $z^\delta$.
Below we use extensively Bochner spaces: for a Banach space $X$ (with norm $\|\cdot\|_X$), we define
$W^{m,p}(0,T;X) = \{v: v(t)\in X\ \mbox{for a.e.}\ t\in(0,T)\ \mbox{and}\ \|v\|_{W^{m,p}(0,T;X)}<\infty \}$.
with  $\|v\|_{W^{m,p}(0,T;X)} = (\sum_{j=0}^m\int_0^T\|u^{(j)}(t)\|_X^p)^\frac{1}{p}$,
The space $L^\infty(0,T;X)$ is defined similarly.

\subsection{The regularized problem and its hybrid approximation}

To recover the coefficient $q$ in the model \eqref{equ:parabolic problem}, we formulate a numerical scheme by
\begin{equation}\label{equ:Tikh-para}
\min_{q\in \mathcal{A}} J_{\gamma}(q)=\frac{1}{2}\|u(q)(t)-z^{\delta}(t)\|^2_{L^2(T_0,T;L^2(\Omega))}+\frac\gamma2\|\nabla q\|_{L^2(\Omega)}^2,
\end{equation}
where $u(t)\equiv u(q)(t)\in H^1_0(\Omega)$ with $u(0)=u_0$ satisfies
\begin{equation}\label{equ:vari-para}	
	(\partial_tu(t),\varphi)+(q\nabla u(t),\nabla\varphi)=(f,\varphi), \quad\forall \varphi\in H^1_0(\Omega),\ \mbox{a.e.}\ t\in(0,T).
\end{equation}

Next we describe the hybrid NN-FEM discretization of problem \eqref{equ:Tikh-para}--\eqref{equ:vari-para}. For the space discretization, we employ NNs and Galerkin FEM to approximate the diffusion coefficient $q$ and state $u$, respectively.
For the time discretization, we employ the backward Euler time-stepping scheme \cite{Thomee:2006}: We divide the time interval $(0,T)$ into $N$ uniform subintervals with a time step size $\tau$ and grid points $t_n=n\tau$, $n=0,\ldots,N$. Next we denote by $v^n=v(t_n)$ and define the backward difference quotient $\bar{\partial}_\tau$ by
$\bar{\partial}_\tau v^n:=\tau^{-1}(v^n-v^{n-1}).$
Further we assume $T_0=N_0\tau$ for some $N_0\in\mathbb{N}$. For a sequence of functions $\{v^n\}_{n={N_0}}^N\subset X$, we defined a discrete norm $\|(v^n)_{N_0}^N\|_{\ell^2(X)}$ by
\begin{equation*}
\|(v^n)_{N_0}^N\|_{\ell^2(X)}=\Big(\tau\sum_{n=N_0}^N\|v^n\|_X^2\Big)^\frac12.
\end{equation*}

With these preliminaries, the hybrid NN-FEM scheme for problem \eqref{equ:Tikh-para}--\eqref{equ:vari-para} reads
\begin{equation}\label{equ:dis-min-para}
	\min_{\theta\in\mathfrak{P}_{p,\epsilon}} J_{\gamma,h,\tau}(q_\theta)=\frac{1}{2}\|(U^n_h(\P01(q_\theta))-z_n^{\delta})_{N_0}^N\|^2_{\ell^2(L^2(\Omega))}+\frac\gamma2\|\nabla q_\theta\|_{L^2(\Omega)}^2,
\end{equation}
where $z_n^\delta:=\tau^{-1}\int_{t_{n-1}}^{t_n} z^\delta(t) \mathrm{d}t$, and $U_h^n\equiv U^n_h(q_\theta)\in X_h$ with $U_h^0(q_\theta)=P_hu_0$ satisfies
\begin{equation}\label{equ:dis-vari-para}	
	(\bar{\partial}_\tau U_h^n,\varphi_h)+(\P01(q_\theta)\nabla U^n_h,\nabla\varphi_h)=(f(t_n),\varphi_h), \quad\forall \varphi_h\in X_h,\ n=1,\ldots,N.
\end{equation}
For any $\gamma>0$,
a standard argument yields the well-posedness of problems \eqref{equ:Tikh-para}--\eqref{equ:vari-para} and \eqref{equ:dis-min-para}--\eqref{equ:dis-vari-para}; Let $q_\theta^*$ be the NN realization of a minimizer $\theta^*$ to problem \eqref{equ:dis-min-para}--\eqref{equ:dis-vari-para}. See the work \cite{KeungZou:1998} for relevant discussions on the pure FEM approximation. It also includes a detailed convergence analysis of the FEM approximation to a global minimizer of problem \eqref{equ:Tikh-para}--\eqref{equ:vari-para} as the discretization parameters $h,\tau\to0^+$. See also \cite{wang2010error,jin2021error} for relevant error analysis.

\subsection{Error analysis}
Now we provide an error analysis of the approximation $\P01(q^*_\theta)$, under the following assumption.
\begin{assumption}\label{assum: para}
For some $p\ge\max(2,d+\mu)$ with $\mu>0$,
$q^\dag\in W^{2,p}(\Omega)\cap \mathcal{A}$, $u_0\in H^2(\Omega)\cap H_0^1(\Omega)\cap W^{1,\infty}(\Omega)$ and $f\in L^{\infty}(0,T;L^{\infty}(\Omega))\cap C^1([0,T];L^2(\Omega))\cap W^{2,1}(0,T;L^2(\Omega))$.
\end{assumption}
Under Assumption \ref{assum: para}, the following regularity estimates hold on $u^\dag\equiv u(q^\dag)$ \cite[p. 128]{jin2021error}: for any $r,q\in(1,\infty)$
\begin{align}
&\partial_t u^\dag\in L^r(0,T;L^q(\Omega)), \Delta u^\dag \in L^r(0,T;L^{q}(\Omega))\mbox{ and } u^\dag\in L^\infty(0,T;W^{1,\infty}(\Omega));\label{eqn:reg-para1}\\
&\|u^\dag(t)\|_{H^2(\Omega)}+\|\partial_t u^\dag(t)\|_{L^2(\Omega)}+t\|\partial_{tt}u^\dag(t)\|_{L^2(\Omega)}\leq c,\mbox{ a.e. }t\in(0,T].\label{eqn:reg-para2}
\end{align}

The next lemma gives the existence of an approximant in the discrete admissible set $\mathfrak{P}_{p,\epsilon}$.
\begin{lemma}\label{lemma:err uq-uh in para}
Let Assumption \ref{assum: para} hold. Then for $\epsilon>0$, there exists  $\theta_\epsilon\in\mathfrak{P}_{p,\epsilon}$ such that
\begin{equation*}
\|(u^\dag(t_n)-U^n_h(\P01(q_{\theta_\epsilon})))_1^N\|^2_{\ell^2(L^2(\Omega))}\leq  c(\tau^2+h^4+\epsilon^2).
\end{equation*}
\end{lemma}
\begin{proof}
By  the argument of Lemma \ref{lem:uq-uh}, we can find  $\theta_\epsilon\in\mathfrak{P}_{p,\epsilon}$  such that
the estimates \eqref{eqn:qeps-W1p} and \eqref{eqn:qeps-W1p-1} hold.
Next we bound $\varrho^n_h:=U^n_h(\P01(q_{\theta_\epsilon}))-U^n_h(q^\dagger)$. It follows from the weak formulations of $U_h^n(\P01(q_{\theta_\epsilon}))$ and $U_h^n(q^\dag)$, cf. \eqref{equ:dis-vari-para}, that $\varrho_h^n$ satisfies $\varrho_h^0=0$ and
\begin{equation*}
(\bar{\partial}_\tau \varrho_h^n,\varphi_h)+(\P01(q_{\theta_\epsilon})\nabla \varrho^n_h,\nabla\varphi_h)=\big((q^\dagger-\P01(q_{\theta_\epsilon}))\nabla U_h^n(q^\dagger),\nabla\varphi_h\big), \quad\forall \varphi_h\in X_h,  \,\,n=1,2,\ldots,N.
\end{equation*}
Setting $\varphi_h=2\varrho_h^n$ into this identity, and then applying H\"{o}lder's inequality lead to
\begin{align*}
&\tau^{-1}(\|\varrho_h^n\|^2_{L^2(\Omega)}-\|\varrho_h^{n-1}\|^2_{L^2(\Omega)})+ 2c_0\|\nabla \varrho_h^n\|_{L^2(\Omega)}^2
\leq 2\|q^\dag-\P01(q_{\theta_\epsilon})\|_{L^\infty(\Omega)}\|\nabla U_h^n(q^\dag)\|_{L^2(\Omega)}\|\nabla \varrho_h^n\|_{L^2(\Omega)}.
\end{align*}
Summing the inequality over $n$ from $1$ to $N$, noting $\varrho_h^0=0$ and applying    \eqref{eqn:qeps-W1p-1} give
\begin{align*}
\|\varrho_h^N\|^2_{L^2(\Omega)}+2c_0\|(\nabla \varrho_h^n)_1^N\|_{\ell^2(L^2(\Omega))}^2&\leq 2\|q^\dag-\P01(q_{\theta_\epsilon})\|_{L^\infty(\Omega)}\tau\sum_{n=1}^N\|\nabla U_h^n(q^\dagger)\|_{L^2(\Omega)}\|\nabla\varrho_h^n\|_{L^2(\Omega)}\\
&\leq c\epsilon\|(\nabla U_h^n(q^\dag))_1^N\|_{\ell^2(L^2(\Omega))}\|(\nabla \varrho_h^n)_1^N\|_{\ell^2(L^2(\Omega))}.
\end{align*}
Since $\|(\nabla U_h^n(q^\dagger))_1^N\|_{\ell^2(L^2(\Omega))}\leq c$ \cite[Lemma 6.2]{wang2010error}, we obtain
$\|(\varrho_h^n)_1^N\|_{\ell^2(H^1(\Omega))} \leq c\epsilon.$
This and the estimate $\|(u^\dag(t_n)-U^n_h(q^\dag))_1^N\|^2_{\ell^2(L^2(\Omega))}\leq  c(\tau^2+h^4)$ \cite[Lemma 4.2]{jin2021error}
complete the proof.
\end{proof}

The next lemma gives an important \textit{a priori} bound.
\begin{lemma}\label{lem:pri grad qN* in para}
Let Assumption \ref{assum: para} hold. For any $\epsilon>0$, let $\theta^*\in\mathfrak{P}_{p,\epsilon}$ be a minimizer to problem \eqref{equ:dis-min-para}-\eqref{equ:dis-vari-para} and $q^*_\theta$ its NN realization. Then the following estimate holds
 \begin{equation*}
 \|(u^\dag(t_n)-U_h^n(\P01(q^*_\theta)))_{N_0}^N\|^2_{\ell^2(L^2(\Omega))}+\gamma\|\nabla \P01(q^*_\theta)\|^2_{L^2(\Omega)}\leq c(\tau^2+h^4+\epsilon^2+\delta^2+\gamma).
\end{equation*}
\end{lemma}
\begin{proof}
Let $q_{\theta_\epsilon}$ be the NN realization of a parameter $\theta_\epsilon \in \mathfrak{P}_{p,\epsilon}$ satisfying \eqref{eqn:qeps-W1p} and \eqref{eqn:qeps-W1p-1}, which implies also $\|q_{\epsilon}\|_{H^1(\Omega)}\leq c$. Under Assumption \ref{assum: para}, the following estimate holds \cite[Lemma 4.1]{jin2021error}
\begin{equation*}		
\|(u^\dag(t_n)-z^\delta_n)_{N_{0}}^N\|^2_{\ell^2(L^2(\Omega))}\leq c(\tau^2+\delta^2).
\end{equation*}
Then by Lemma \ref{lemma:err uq-uh in para} and the minimizing property of $q_\theta^*$, i.e.,
$J_{\gamma,h,\tau}(q^*_\theta)\leq  J_{\gamma,h,\tau}(q_{\theta_\epsilon})$, we derive
\begin{align*}		
&\|(U^n_h(\P01(q^*_\theta))-z_n^\delta)_{N_{0}}^N\|^2_{\ell^2(L^2(\Omega))}+\gamma\|\nabla q^*_\theta\|^2_{L^2(\Omega)}
\leq \|(U^n_h(\P01(q_{\theta_\epsilon}))-z^\delta_n)_{N_{0}}^N\|^2_{\ell^2(L^2(\Omega))}+\gamma\|\nabla q_{\theta_\epsilon}\|^2_{L^2(\Omega)} \\
\leq  & c\big(\|(U^n_h(\P01(q_{\theta_\epsilon}))-u^\dag(t_n))_{N_0}^N\|^2_{\ell^2(L^2(\Omega))}+\|(u^\dag(t_n)-z^\delta_n)_{N_{0}}^N\|^2_{\ell^2(L^2(\Omega))}+\gamma\big)
\leq c(\tau^2+h^4+\epsilon^2+\delta^2+\gamma),
\end{align*}
Then by the triangle inequality, we have
\begin{align*}
\|(u^\dag(t_n)-&U^n_h(\P01(q^*_\theta)))_{N_{0}}^N\|^2_{\ell^2(L^2(\Omega))}+\gamma\|\nabla q^*_\theta\|^2_{L^2(\Omega)}
\leq c\|(u^\dagger(t_n)-z_n^\delta)_{N_{0}}^N\|^2_{\ell^2(L^2(\Omega))}\\
&+c\|(z_n^\delta-U^n_h(\P01(q^*_\theta)))_{N_{0}}^N\|^2_{\ell^2(L^2(\Omega))}
 +\gamma\|\nabla q^*_\theta\|^2_{L^2(\Omega)}
\leq c(\tau^2+h^4+\epsilon^2+\delta^2+\gamma).
\end{align*}
Finally, the bound on $\| \nabla \P01(q_\theta^*) \|_{L^2(\Omega)}$ follows from the stability estimate \eqref{eqn:P01-stab}.
\end{proof}

Now we can state an error estimate on the NN  approximation $q_\theta^*$.
\begin{theorem}\label{thm: error in para}
Let Assumption \ref{assum: para} hold.
Fix any $\epsilon>0$, and let $\theta^*\in\mathfrak{P}_{p,\epsilon}$ be a minimizer to problem \eqref{equ:dis-min-para}-\eqref{equ:dis-vari-para} and $q^*_\theta$ its NN realization.
Then with $\eta^2:=\tau^2+h^4+\epsilon^2+\delta^2+\gamma$, there holds
\begin{align*}
&\quad \tau^3\sum_{j=N_{0}+1}^N\sum_{i=N_{0}+1}^j\sum_{n=i}^j\int_{\Omega}\Big(\frac{q^\dag-\P01(q_\theta^*)}{q^\dag}\Big)^2\big(q^\dag |\nabla u^\dagger(t_n)|^2+\big(f(t_n)-\partial_tu^\dagger(t_n)\big)u^\dagger(t_n)\big)\ \mathrm{d}x\\
&\leq c( \min(h^{-1}\eta + h,1) + h \gamma^{-\frac12}\eta)
\gamma^{-\frac12}\eta.
\end{align*}	
\end{theorem}
\begin{proof}
For any $\varphi\in H_0^1(\Omega)$, the weak formulations of $u^\dag$ and $U_h^n(\P01(q_\theta^*)))$in \eqref{equ:vari-para} and \eqref{equ:dis-vari-para} yield
\begin{align*}
&\quad \big((q^\dag-\P01(q_\theta^*))\nabla u^\dagger(t_n),\nabla\varphi\big) \\
& = \big ((q^\dag-\P01(q_\theta^*))\nabla u^\dagger(t_n),\nabla(\varphi-P_h\varphi)\big) + \big( (q^\dag-\P01(q_\theta^*))\nabla u^\dagger(t_n),\nabla P_h\varphi\big) \\
& = \big((q^\dag-\P01(q_\theta^*))\nabla u^\dagger(t_n),\nabla(\varphi-P_h\varphi)\big)+\big(\P01(q_\theta^*)\nabla(U_h^n(\P01(q_\theta^*))-u^\dagger(t_n)),\nabla P_h\varphi\big)\\
&\quad +\big(q^\dag \nabla u^\dagger(t_n)-\P01(q_\theta^*)\nabla U_h^n(\P01(q_\theta^*)),\nabla P_h\varphi\big) \\
& = -\big(\nabla\cdot\big((q^\dag-\P01(q_\theta^*))\nabla u^\dagger(t_n)\big),\varphi-P_h\varphi\big)+\big(\P01(q_\theta^*)\nabla(U_h^n(\P01(q_\theta^*))-u^\dagger(t_n)),\nabla P_h\varphi\big)\\  &\quad+\big(\bar{\partial}_\tau U_h^n(\P01(q_\theta^*))-\partial_t u^\dagger(t_n), P_h\varphi\big)=: {\rm I}^n + {\rm II}^n + {\rm III}^n.
\end{align*}
Next we set $\varphi\equiv\varphi^n=\frac{q^\dag-\P01(q_\theta^*)}{q^\dag}u^\dagger(t_n)$ in the identity, and bound the three terms separately. Under Assumption \ref{assum: para},
the regularity bound \eqref{eqn:reg-para2} and the box constraint $\P01(q_\theta^*)\in\mathcal{A}$ imply
\begin{equation}\label{eqn:est-fym}
\max_{0\le n \le N}\|\nabla \varphi^n\|_{L^2(\Omega)}\leq c(1+\|\nabla  \P01(q_\theta^*)\|_{L^2(\Omega)}).
\end{equation}
Then repeating the argument for Theorem \ref{thm:error-ellip} and applying Lemma \ref{lem:pri grad qN* in para} lead to
\begin{equation*}
|{\rm I}^n|\leq ch(1+\|\nabla\P01(q_\theta^*)\|_{L^2(\Omega)})\|\nabla \varphi^n\|_{L^2(\Omega)}\leq ch(1+\|\nabla \P01(q_\theta^*)\|^2_{L^2(\Omega)})\leq ch\gamma^{-1}\eta^2.
\end{equation*}
Next, by the Cauchy--Schwarz inequality, the $H^1(\Omega)$-stability of $P_h$, the box constraint $\P01(q_\theta^*)\in\mathcal{A}$ and the estimate \eqref{eqn:est-fym}, we bound the term ${\rm II}^n$ as
\begin{align*}
 |{\rm II}^n|& \le c   \|\nabla(U_h^n(\P01(q_\theta^*))-u^\dag(t_n))\|_{L^2(\Omega)} \|  \nabla P_h\varphi^n\|_{L^2\II}\\
 &\le  c   \|\nabla(U_h^n(\P01(q_\theta^*))-u^\dag(t_n))\|_{L^2(\Omega)} \|  \nabla  \varphi^n\|_{L^2\II} \\
 &\le  c  (1+\| \nabla\P01(q_\theta^*) \|_{L^\II}) \|\nabla(U_h^n(\P01(q_\theta^*))-u^\dag(t_n))\|_{L^2(\Omega)} .
\end{align*}
Then it follows from the inverse estimate in the space $X_h$ \cite[(1.12), p. 4]{Thomee:2006}
and \eqref{inequ: L2 proj approx} that
\begin{align*}
&\quad\tau\sum_{n=N_{0}}^N|{\rm II}^n| \le   c  \gamma^{-\frac12}\eta \big\|(\nabla(U_h^n(\P01(q_\theta^*))-u^\dag(t_n)))_{N_0}^N\big\|_{\ell^2(L^2(\Omega))} \\
&\le c\gamma^{-\frac12}\eta \big(\big\|(\nabla(U_h^n(\P01(q_\theta^*))-P_h u^\dag(t_n)))_{N_0}^N\big\|_{\ell^2(L^2(\Omega))} + \big\|(\nabla(u^\dag(t_n)-P_h u^\dag(t_n)))_{N_0}^N\big\|_{\ell^2(L^2(\Omega))} \big) \\
&\le   c  \gamma^{-\frac12}\eta \big(h^{-1}\| (U_h^n(\P01(q_\theta^*))-P_hu^\dag(t_n))_{N_0}^N\|_{\ell^2(L^2(\Omega))} + h\|(u^\dag(t_n))_{N_0}^N\|_{\ell^2(H^2(\Omega))} \big).
\end{align*}
Now by applying the $L^2(\Omega)$-stability of $P_h$ and Lemma \ref{lem:pri grad qN* in para}, we deduce
 \begin{align*}
\tau\sum_{n=N_{0}}^N|{\rm II}^n|
&\le  c  \gamma^{-\frac12}\eta \big( h^{-1}\|(U_h^n(\P01(q_\theta^*))-  u^\dag(t_n))_{N_{0}}^N\|_{\ell^2(L^2(\Omega))} + h \big) \le c(h  + h^{-1}\eta) \gamma^{-\frac12} \eta.
\end{align*}
Meanwhile, the box constraint $\P01(q_\theta^*)\in\mathcal{A}$ and the standard energy argument imply
\begin{align}\label{eqn:nablauhm-L2}
\|(\nabla(U_h^n(\P01(q_\theta^*))-u^\dag(t_n)))_{N_{0}}^N\|_{\ell^2(L^2(\Omega))}\le c.
\end{align}
Thus we obtain
\begin{equation*}
\tau\sum_{n=N_{0}}^N|{\rm II}^n|
 \le  c \gamma^{-\frac12} \eta \min(1, h  + h^{-1}\eta).
\end{equation*}
For the last term ${\rm III}^n$, we further split it into
\begin{equation*}
{\rm III}^n=\big(\bar{\partial}_\tau U_h^n(\P01(q_\theta^*))-\bar{\partial}_\tau u^\dagger(t_n), P_h\varphi^n\big)+\big( \bar{\partial}_\tau u^\dagger(t_n)-\partial_t u^\dagger(t_n), P_h\varphi^n\big)=: {\rm III}_{1}^n + {\rm III}_{2}^n,
\end{equation*}
and then bound ${\rm III}_{1}^n$ and ${\rm III}_{2}^n$ separately.
Repeating the argument in \cite[Theorem 4.5]{jin2021error} gives
\begin{equation*}  \bigg|\tau^3\sum_{j=N_{0}+1}^N\sum_{i=N_{0}+1}^j\sum_{n=i}^j{\rm III}_{2}^n\bigg|\leq c\tau .
\end{equation*}
To bound the term ${\rm III}_{1}^n$, by the summation by parts formula, we deduce
\begin{align*}
\tau\sum_{n=i}^j{\rm III}_{1}^n=&-\tau\sum_{n=i}^{j-1}\big(U_h^n(\P01(q_\theta^*))-u^\dagger(t_n),\bar{\partial}_\tau P_h\varphi^{n+1}\big) +\big(U_h^j(\P01(q_\theta^*))-u^\dagger(t_j),P_h\varphi^j\big)\\
&-\big(U_h^{i-1}(\P01(q_\theta^*))-u^\dagger(t_{i-1}),P_h\varphi^i\big).
\end{align*}
Since $\|P_h\varphi^n\|_{L^2(\Omega)}\leq \| \varphi^n\|_{L^2\II} \le c$, we get
\begin{equation*}
\bigg|\tau\sum_{j=N_{0}+1}^N\big(U_h^j(\P01(q_\theta^*))-u^\dagger(t_j),P_h\varphi^j\big) \bigg| +
\bigg| \tau\sum_{i=N_{0}+1}^N\big(U_h^{i-1}(\P01(q_\theta^*))-u^\dagger(t_{i-1}),P_h\varphi^i\big)\bigg|\leq c \eta.
\end{equation*}
Moreover, from the $L^2(\Omega)$ stability of $P_h$, Assumption \ref{assum: para} and the box constraint $\P01(q_\theta^*)\in\mathcal{A}$, we deduce
\begin{align*}
\|\bar{\partial}_\tau P_h\varphi^n\|_{L^2(\Omega)}&\leq\tau^{-1}\Big\|\int_{t_{n-1}}^{t_n}\frac{q^\dagger-\P01(q_\theta^*)}{q^\dagger}\partial_{t}u(t)\ \mathrm{d}t\Big\|_{L^2(\Omega)}
\leq c \| \partial_t u \|_{C([t_{n-1},t_n];L^2\II)}.
\end{align*}
Thus, we have
\begin{equation*}
\Big|\tau^3\sum_{j=N_{0}+1}^N\sum_{i=N_{0}+1}^j\sum_{n=i}^j\big(U_h^n(q_\theta^*)-u^\dag(t_n),\bar{\partial}_\tau P_h\varphi^{n+1}\big) \Big|\leq c\eta.
\end{equation*}
Finally, combining the preceding estimates with the identify
\begin{equation*}
\big((q^\dag-\P01(q_\theta^*))\nabla u^\dagger(t_n),\nabla\varphi^n\big)
=\frac12\int_{\Omega}\Big(\frac{q^\dag-\P01(q_\theta^*)}{q^\dag}\Big)^2\big(q^\dag |\nabla u^\dagger(t_n)|^2+\big(f(t_n)-\partial_tu^\dagger(t_n)\big)u^\dagger(t_n)\big)\ \mathrm{d}x
\end{equation*}
 completes the proof of the theorem.
\end{proof}

Similarly, we can impose a positivity condition: there exists some $\beta\geq0$ such that for any $t\in[T_0,T]$
\begin{equation}\label{Cond: P para}
q^\dag |\nabla u^\dagger(x,t)|^2+\big(f(x,t)-\partial_tu^\dagger(x,t)\big)u^\dagger(x,t)\geq  c\ {\rm dist}(x,\partial\Omega)^\beta, \quad \mbox{a.e. in}\ \Omega.
\end{equation}
This condition holds with $\beta=0,2$ under suitable assumptions on the problem data \cite[Propositions 4.7 and 4.8]{jin2021error}. Under condition \eqref{Cond: P para}, the argument of Theorem \ref{thm:error-ellip} gives the following $L^2(\Omega)$ error bound.
\begin{corollary}\label{cor:err-para}
Under the assumptions in Theorem \ref{thm: error in para} and condition \eqref{Cond: P para}, there holds
\begin{equation*}
\|q^\dagger-q_\theta^*\|_{L^2(\Omega)}\leq c\big( \min(h^{-1}\eta + h,1) + h \gamma^{-\frac12}\eta)
\gamma^{-\frac12}\eta\big)^{\frac{1}{2(1+\beta)}}.
\end{equation*}
\end{corollary}

\subsection{Quadrature error analysis}
Now we study the influence of quadrature errors on the reconstruction. Like before, we formulate a practical hybrid NN-FEM discretization scheme of problem \eqref{equ:Tikh-para}-\eqref{equ:vari-para} (with numerical integration) by
\begin{equation}\label{equ:dis-min-para-q}
\min_{\theta\in \mathfrak{P}_{\infty,\epsilon}} \tilde{J}_{\gamma,h,\tau}(q_\theta)=\frac{1}{2}\|(\tilde{U}^n_h(\P01(q_\theta))-z_n^{\delta})_{N_0}^M\|^2_{\ell^2(L^2(\Omega))}+\frac\gamma2  Q_h(|\nabla q_\theta|^2) ,
\end{equation}
where $z_n^\delta:=\tau^{-1}\int_{t_{n-1}}^{t_n} z^\delta(t) \mathrm{d}t$, and $\tilde{U}_h^n\equiv \tilde{U}^n_h(\P01(q_\theta))\in X_h$ with $\tilde{U}_h^0(\P01(q_\theta))=P_hu_0$ satisfies
\begin{equation}\label{equ:dis-vari-para-q}	(\bar{\partial}_\tau \tilde{U}_h^n,\varphi_h)+(\P01(q_\theta)\nabla \tilde{U}^n_h,\nabla\varphi_h)_h=(f(t_n),\varphi_h), \quad\forall \varphi_h\in X_h,\ n=1,\ldots,N.
\end{equation}
Using the box constraint  $\P01(q_\theta)\in\mathcal{A}$ and the standard energy argument, we have
\begin{equation}\label{eqn:tildeU-para-quad}
\|(\nabla \tilde U_h^n(\P01(q_\theta)))_1^N\|_{\ell^2(L^2(\Omega))} \le c.
\end{equation}
The existence of a discrete forward map $\P01(q_\theta)\mapsto \{U_h^n\}_{n=1}^N$ follows from the ellipticity of the broken $L^2(\Omega)$ semi-inner product $(\cdot,\cdot)_h$ over the space $X_h$, and a standard argument yields that problem
\eqref{equ:dis-min-para-q}-\eqref{equ:dis-vari-para-q} has at least one minimizer $\tilde \theta^*$ with a continuous dependence on the data. Next we derive (weighted) $L^2(\Omega)$ error bounds of $P_\mathcal{A}(\tilde{q}_\theta^*)$, with the NN realization $\tilde q_\theta^*$ of the minimizer $\tilde{\theta}^*$.

\begin{assumption}\label{assum: para quad}
$q^\dag\in W^{2,\infty}(\Omega)\cap \mathcal{A}$,
$u_0\in H^2(\Omega)\cap H_0^1(\Omega)\cap W^{1,\infty}(\Omega)$
and $f\in L^{\infty}(0,T;L^{\infty}(\Omega))\cap C^1(0,T;L^2(\Omega))\cap W^{2,1}(0,T;L^2(\Omega))$.
\end{assumption}

The next lemma gives an analogue of Lemma \ref{lemma:err uq-uh in para} for the quadrature scheme.
\begin{lemma}\label{lemma:err uq-uh in para, quad}
Let Assumption \ref{assum: para quad} hold.  Then for small $\epsilon>0$, there exists  $\theta_\epsilon\in\mathfrak{P}_{\infty,\epsilon} $ such that
\begin{equation*}
\|\big(u(q^\dag)(t_n)-\tilde{U}^n_h(\P01( {q}_{\theta_\epsilon}))\big)_1^N\|^2_{\ell^2(L^2(\Omega))}\leq c(\tau^2+h^4+\epsilon^2).
\end{equation*}
\end{lemma}
\begin{proof}
It follows from Lemma \ref{lem:uq-uh} that there exists $\theta_\epsilon\in\mathfrak{P}_{\infty,\epsilon}$ such that
the estimate \eqref{eqn:qeps-W1inf} holds for the NN realization $q_{\theta_\epsilon}$.
Let $\varrho_h^n = \tilde U^n_h(q^\dagger)-\tilde U^n_h(\P01(q_{\theta_\epsilon}))$. Then it satisfies
$\varrho_h^n = 0$ and for all $n=1,2,\ldots,N$
$$(\partial_\tau \varrho_h^n, \varphi_h) + (\P01(q_{\theta_\epsilon})\nabla \varrho_h^n, \nabla \varphi_h)_h = ((\P01(q_{\theta_\epsilon}) - q^\dag)\nabla \tilde U_h^n(q^\dag), \nabla \varphi_h)_h,\quad \forall \varphi_h \in X_h.$$
Repeating the  argument of Lemma \ref{lemma:err uq-uh in para}
gives
\begin{equation}\label{eqn:err-uq-para-quad}
\|(\tilde U^n_h(q^\dagger)-\tilde U^n_h(\P01(q_{\theta_\epsilon})))_1^N\|_{\ell^2(L^2(\Omega))}\leq  c\epsilon.
\end{equation}
Since $\|( u^\dag(t_n) - U^n_h(q^\dagger))_1^N \|_{\ell^2(L^2(\Omega))}\leq c(\tau + h^2)$  \cite[Lemma 4.2]{jin2021error},
it suffices to show
\begin{equation}\label{eqn:est-para-h}
\|( U^n_h(q^\dagger) - \tilde U^n_h(q^\dagger) )_1^N \|_{\ell^2(L^2(\Omega))}\leq c h^2.
\end{equation}
Let $e_h^n = U^n_h(q^\dagger) - \tilde U_h^n(q^\dag)$. Then $e_h^n$
satisfies $e_h^n =0$ and
\begin{equation*}
\begin{aligned}
(\partial_\tau e_h^n, \varphi_h) + (q^\dag \nabla e_h^n, \nabla \varphi_h)_h = (q^\dag \nabla U^n_h(q^\dagger), \nabla \varphi_h)_h -(q^\dag \nabla U^n_h(q^\dagger), \nabla \varphi_h),\quad\forall\varphi_h\in X_h, n=1,\ldots,N.
\end{aligned}
\end{equation*}
Now upon choosing $\varphi_h = e_h^n$ and applying Lemma \ref{lem:quad-error} (with $p=2$), we obtain
$$|(q^\dag \nabla e^n_h, \nabla e_h^n)_h -(q^\dag \nabla e^n_h, \nabla e_h^n)| \le  c h^2  \|q^\dag\|_{W^{2,\infty}(\Omega)}\|\nabla U_h^n(q^\dagger)\|_{L^2(\Omega)}\|\nabla e_h^n\|_{L^2(\Omega)}. $$
Consequently, we have
\begin{equation*}
\begin{aligned}
\tfrac12 \partial_\tau \|e_h^n\|_{L^2\II}^2 + c_0 \|\nabla e_h^n\|_{L^2\II}^2 \le  c h^2\|q^\dagger\|_{W^{2,\infty}(\Omega)}\|\nabla U_h^n(q^\dagger)\|_{L^2(\Omega)}\|\nabla e_h^n\|_{L^2(\Omega)}.
\end{aligned}
\end{equation*}
Then upon summing the identity over $n$ from $1$ to $N$, noting $e_h^0=0$, we arrive at
\begin{equation*}
\begin{aligned}
\|e_h^N\|_{L^2\II}^2 + c_0 \|(\nabla e_h^n)_1^N\|_{\ell^2(L^2\II)}^2 &\le  c h^2 \|(\nabla U_h^n(q^\dagger))_{1}^N\|_{\ell^2(L^2(\Omega))}\|(\nabla e_h^n)_1^N\|_{\ell^2(L^2(\Omega))}.
\end{aligned}
\end{equation*}
Then the estimate \eqref{eqn:est-para-h} follows from the bound $\|(\nabla U_h^n(q^\dagger))_1^N\|_{\ell^2(L^2(\Omega))}\leq c$ \cite[Lemma 6.2]{wang2010error}.
\end{proof}

The next lemma gives an \textit{a priori} bound on $u(q^\dag)(t_n)-\tilde{U}_h^n(\tilde{q}^*_\theta)$ and $\tilde{q}^*_\theta$, with the quadrature approximation. The proof is identical with that for Lemma \ref{lem:pri grad qN* in para}, and hence omitted.
\begin{lemma}\label{lem:pri grad qN* in para, quad}
Let Assumption \ref{assum: para quad} hold. Fix $\epsilon>0$, and let $\theta^*\in\mathfrak{P}_{\infty,\epsilon}$ be a minimizer to problem \eqref{equ:dis-min-para-q}-\eqref{equ:dis-vari-para-q} and $\tilde q^*_\theta$ its NN realization. Then the following estimate holds
\begin{equation*}
\|(u(q^\dag)(t_n)-\tilde{U}_h^n(\P01(\tilde q^*_\theta)))_{N_0}^N\|^2_{\ell^2(L^2(\Omega))}+\gamma Q_h(|\nabla \P01(\tilde q^*_\theta)|^2)\leq c(\tau^2+h^4+\epsilon^2+\delta^2+\gamma).
\end{equation*}
\end{lemma}

Now we can present the main result of this section.
\begin{theorem}\label{thm: error in para, quad}
Let Assumption \ref{assum: para quad} hold. Fix $\epsilon>0$, and let $\theta^*\in\mathfrak{P}_{\infty,\epsilon}$ be a minimizer to problem \eqref{equ:dis-min-para-q}-\eqref{equ:dis-vari-para-q} and $\tilde q^*_\theta$ its NN realization. Let $\eta^2:=\tau^2+h^4+\epsilon^2+\delta^2+\gamma$ and $ \zeta=1+\gamma^{-1}\eta^2+2^{-2n}h^2 R^{4L}\mathcal{W}^{4L-4}.$
Then the following estimate holds
\begin{align*}
&\tau^3\sum_{j=N_{0}+1}^N\sum_{i=N_{0}+1}^j\sum_{n=i}^j\int_{\Omega}\bigg(\frac{q^\dag-\P01(\tilde{q}_\theta^*)}{q^\dag}\bigg)^2\big(q^\dag |\nabla u^\dag(t_n)|^2+\big(f(t_n)-\partial_tu^\dag(t_n)\big)u^\dag(t_n)\big)\ \mathrm{d}x\\
&\leq c\big(h\zeta^\frac12+ (\min(h^{-1}\eta+h,1) +  2^{-n}hd^\frac{1}{2} R^{L}{W}^{L} \big)\zeta^\frac12.
\end{align*}
\end{theorem}
\begin{proof}
For any $\varphi\in H_0^1(\Omega)$, the weak formulations of $u^\dag$ and $\tilde{U}_h^m(\tilde q_\theta^*)$, cf. \eqref{equ:vari-para} and \eqref{equ:dis-vari-para-q}, imply
\begin{align*}
 &\quad\big((q^\dag- \P01(\tilde q_\theta^*))\nabla u^\dagger(t_n),\nabla\varphi\big) \\
 &= \big ((q^\dag-\P01(\tilde q_\theta^*))\nabla u^\dagger(t_n),\nabla(\varphi-P_h\varphi)\big) + \big( (q^\dag-\P01(\tilde q_\theta^*))\nabla u^\dagger(t_n),\nabla P_h\varphi\big) \\
		& = -\big(\nabla\cdot\big((q^\dag-\P01(\tilde q_\theta^*))\nabla u^\dag(t_n)\big),\varphi-P_h\varphi\big)+\big(\P01(\tilde q_\theta^*)\nabla(\tilde{U}_h^n(\P01(\tilde q_\theta^*)) -u^\dag(t_n)),\nabla P_h\varphi\big)\\
  &\quad+ \big(\bar{\partial}_\tau \tilde{U}_h^n(\P01(\tilde{q}_\theta^*))-\partial_t u^\dag(t_n), P_h\varphi\big)\\
  &\quad+ \Big(\big(\P01(\tilde q_\theta^*)\nabla \tilde{U}_h^n(\P01(\tilde q_\theta^*)),\nabla P_h\varphi\big)_h-\big(\P01(\tilde q_\theta^*) \tilde{U}_h^n(\P01(\tilde q_\theta^*)),\nabla P_h\varphi\big)\Big)
 =: {\rm I}^n + {\rm II}^n + {\rm III}^n + {\rm IV}^n.
\end{align*}
Set $\varphi\equiv\varphi^n=\frac{q^\dag-\P01(\tilde q_\theta^*)}{q^\dag}u^\dag(t_n)$ in the identity.
Lemma \ref{lem:pri grad qN* in para, quad} and the argument for \eqref{eqn:nablaphi-quad} imply
\begin{equation*}
\|\nabla \varphi\|_{L^2(\Omega)}\leq c(1+\|\nabla  \P01(\tilde q_\theta^*)\|_{L^2(\Omega)})
\leq c\zeta^\frac12.
\end{equation*}
Then repeating the argument for Theorem \ref{thm: error in para} yields
\begin{align*}
&|{\rm I}^n| \leq ch\zeta, \quad
\tau\sum_{m=M_{0}}^M|{\rm II}^n|\leq  c \min(h^{-1}\eta+h,1) \zeta^\frac12\quad\mbox{and}\quad
\tau^3\sum_{j=N_{0}+1}^N\sum_{i=N_{0}+1}^j\sum_{n=i}^j{\rm III}^n\leq c\eta.
\end{align*}
Then it follows from Lemma \ref{lem:quad-error} (with $p=1$), the stability of $P_\mathcal{A}$ and the bounds \eqref{eqn:dvthe} and \eqref{eqn:tildeU-para-quad} that
\begin{align*}
\tau\sum_{j=N_0}^N|{\rm IV}^n|&\leq c2^{-n}h\|\P01(\tilde q_\theta^*)\|_{W^{1,\infty}(\Omega)}
\Big(\max_{N_0\le n\le N}\|\nabla P_h \varphi^n\|_{L^2(\Omega)}\Big)
\big\|(\nabla \tilde U_h^n(\P01(\tilde q_\theta^*)))_{N_0}^N\big\|_{\ell^2(L^2(\Omega))}\\
&\leq  c 2^{-n}h  \zeta^{\frac12}
(\|  \P01(\tilde q_\theta^*) \|_{L^\infty\II} + \| \nabla \P01(\tilde q_\theta^*) \|_{L^\infty\II}) \big\|(\nabla \tilde U_h^n(\P01(\tilde q_\theta^*)))_{N_0}^N \|_{\ell^2(L^2(\Omega))}\\
&\leq  c 2^{-n}h \zeta^{\frac12}
(1 + \| \nabla\tilde q_\theta^* \|_{L^\infty\II}) \big\|(\nabla \tilde U_h^n(\P01( \tilde q_\theta^*)))_{N_0}^N\|_{\ell^2(L^2(\Omega))}\leq c 2^{-n}h d^\frac{1}{2}R^LW^{L} \zeta^\frac12.
\end{align*}
Combining the preceding estimates directly shows the the desired assertion.
\end{proof}

\section{Numerical results and discussions}
\label{sec: numerics}

\subsection{Numerical implementation}
First we describe the implementation of the hybrid NN-FEM approach, i.e., problems \eqref{equ:dis-min-ellip}-\eqref{equ:dis-vari-ellip} and  \eqref{equ:dis-min-para}-\eqref{equ:dis-vari-para}. We train the NNs by minimizing the losses \eqref{equ:dis-min-ellip} and \eqref{equ:dis-min-para} for the elliptic and parabolic cases, respectively.
Traditionally, the NNs are trained using gradient type methods and the gradient is computed using back-propagation \cite{LeCun:1988}, which can be done in many software framework, e.g., PyTorch and Tensorflow. In the hybrid method, we employ the adjoint technique \cite{Cea:1986}. By the chain rule, the gradient of the loss $J_{\gamma}$ to the NN parameter $\theta$ is given by $\frac{\d J_{\gamma}}{\d\theta}=\frac{\d J_{\gamma}}{\d q}\frac{\d q}{\d\theta}$. We compute $J'(q)=\frac{\d J_{\gamma}}{\d q} $ using the standard adjoint technique, and $\frac{\d q}{\d\theta} $ using back-propagation.
\begin{lemma}\label{lem:gradient_J}
Let $J_{\gamma}$ be the functional in \eqref{equ:Tikhonov problem in ellip} or \eqref{equ:Tikh-para}. The gradient of $J_{\gamma}$ at $q\in H^1(\Omega)$ is given by
\begin{align*}
 J_{\gamma}'(q) & =\left\{\begin{aligned}
    \nabla u\cdot\nabla v-\gamma\Delta q, & \quad \mbox{elliptic},\\
    \int_{T_0}^T \nabla u\cdot\nabla w\ \mathrm{d}t-\gamma\Delta q,& \quad \mbox{parabolic},
 \end{aligned}\right.
\end{align*}
where $v$ and $w$ respectively solve
\begin{equation*}
\left\{
        \begin{aligned}
        -\nabla\cdot(q\nabla v) &= z^{\delta}-u, &&\mbox{in } \Omega, \\
		v&=0, &&\mbox{on } \partial\Omega,
        \end{aligned}
        \right.\quad \mbox{and}\quad  \left\{
        \begin{aligned}
        -\partial_t w-\nabla\cdot(q\nabla w)&=z^{\delta}-w,  &&\mbox{in } \Omega\times(T_0,T),\\
        w&=0,  &&\mbox{on } \partial\Omega\times(T_0,T),\\
        w(T)&=0,  &&\mbox{in } \Omega.
        \end{aligned}
        \right.
\end{equation*}
\end{lemma}
\begin{proof}
We only show the elliptic case, since the proof for the parabolic case is similar. The derivative $J'_\gamma$ of $J_\gamma$ at $q$ along the direction $p$ is given by
$J_\gamma'(p)[q] = (u(q)-z^{\delta},u'(q)[p]) + \gamma (\nabla q,\nabla p)$,
where $u'(q)[p] $ is the derivative of $u(q)$ at $q$ along the direction $p$. Note that $u'(q)[p]$ satisfies
\begin{align}\label{eqn:weak u_prime_q_p}
  (q\nabla u'(q)[p], \nabla\varphi)=-(p\nabla u, \nabla\varphi),\quad \forall \varphi \in H_0^1(\Omega).
\end{align}
Meanwhile, the weak formulation of $v$ is given by
\begin{align}\label{eqn:weak_adjoint}
    (q\nabla v, \nabla\varphi) = (z^{\delta}-u,\varphi),\quad \forall \varphi\in H_0^1(\Omega).
\end{align}
Now choosing $\varphi =v$ in \eqref{eqn:weak u_prime_q_p} and $\varphi=u'(q)[p]$ in \eqref{eqn:weak_adjoint} and subtracting the resulting identities yield
$(u(q)-z^{\delta}), u'(q)[p]) =(p\nabla u,\nabla v)$. This shows the desired assertion directly.
\end{proof}

Note that the gradient $J_{\gamma}'(q)$ belongs to $(H^{1}(\Omega))^*$, which is a distribution and unsuitable for updating the coefficient $q$ directly. To remedy this issue, we apply the Riesz map to pull it back to the space $H^1(\Omega)$. This gives a sufficiently regular gradient for updating $q$. By  Riesz representation theorem, there exists $G\in H^1(\Omega)$ such that
\begin{equation*}
\langle J_{\gamma}'(q),\varphi\rangle_{(H^{1}(\Omega))^*,H^1(\Omega)}=(G,\varphi)_{H^1(\Omega)},\quad \forall \varphi\in H^1(\Omega),
\end{equation*}
where $\langle \cdot,\cdot \rangle_{(H^{1}(\Omega))^*,H^1(\Omega)}$  denotes the duality pairing between $(H^{1}(\Omega))^*$ and $H^1(\Omega)$, and $ (\cdot,\cdot)_{H^1(\Omega)}$ the $H^1(\Omega)$ inner product. Thus, $G\in H^1(\Omega)$ is the weak solution of the following elliptic problem:
\begin{equation*}
    \left\{
    \begin{aligned}
        -\Delta G+ G&= J_{\gamma}'(q),\quad \mbox{in }\Omega,\\
        \partial_{\nu} G&=0, \quad \mbox{on }\partial\Omega.
    \end{aligned}
    \right.
\end{equation*}

\subsection{Numerical experiments and discussions}

Now we present numerical reconstructions $P_\mathcal{A}(q_\theta^*)$ and $q_h$ using the hybrid NN-FEM approach and the fully FEM. Their accuracy to the exact diffusivity $q^\dag$ is measured by the relative error:
\begin{equation*}
    e(\P01(q_\theta^*)):=\|q^\dag-\P01(q_\theta^*)\|_{L^2(\Omega)}/\|q^{\dag}\|_{L^2(\Omega)}\quad \mbox{and}\quad e(q_h^*):=\|q^\dag-q_h^*\|_{L^2(\Omega)}/\|q^{\dag}\|_{L^2(\Omega)}.
\end{equation*}
The exact data $u^\dag$ is generated on a finer mesh, and in the elliptic case, the noisy data $z^\delta$ is generated by
$z^\delta(x) = u^\dag(x) + \epsilon\|u^\dag\|_{L^\infty(\Omega)} \xi(x)$ for $x\in \Omega$,
where $\xi(x)$ follows the standard Gaussian distribution, and $\epsilon>0$ is the (relative) noise level. The parabolic case is similar. Unless otherwise stated, the NN for approximating $q$ is taken to be $d$-32-32-1 (i.e., with two hidden layers, each having 32 neurons). The mesh size $h$ is 1/40 and 1/32 and the time step size $\tau$ is $1/1000$ and $1/200$, for one- and two-dimensional problems, respectively. These FEM discretization parameters are applied to both hybrid method and pure FEM. The resulting loss is minimized using ADAM \cite{KingmaBa:2015}. All the experiments were carried out on a personal desktop  (with Windows 10, with RAM 64.0GB, Intel(R) Core(TM) i9-10900 CPU,
2.80 GHz). The hybrid NN-FEM approach was implemented with Python 3.8.8 on the software framework TensorFlow using the SciKit-fem \cite{skfem2020} package to solve the PDEs, and the pure FEM approach was implemented on MATLAB 2022a. Unless otherwise stated, the level of quadrature is fixed at $n=0$ (i.e., no further sub-division).

\begin{table}[htp!]
  \centering
    \caption{The relative errors for the examples at different noise levels.\label{tab:exam}}
    \begin{tabular}{c|ccccc|ccccc|}
    \toprule
    \multicolumn{1}{c}{}&
    \multicolumn{5}{c}{(a) Example \ref{ex:ell}(i)}&\multicolumn{5}{c}{(b) Example \ref{ex:ell}(ii)}\\
    \cmidrule(lr){2-6} \cmidrule(lr){7-11}
    $\epsilon$   & 10e-2 & 5e-2 & 1e-2 & 5e-3 &1e-3  & 10e-2 & 5e-2 & 1e-2 & 5e-3 &1e-3\\
    \midrule
    $\gamma_{\theta}$ & 1e-6 & 1e-6 & 1e-6 & 1e-7 & 1e-7 & 1e-7 & 1e-7 & 1e-8 & 1e-8 & 1e-8\\
    $e(q_\theta^*)$ & 3.17e-2 & 2.25e-2 & 1.24e-2 & 1.24e-2 & 1.12e-2 & 8.92e-2 & 4.76e-2 & 3.86e-2 & 3.91e-2 & 2.67e-2\\
    $\gamma_{h}$ & 2e-6 & 1e-6 & 1e-7 & 5e-8 & 1e-8 & 2e-6 & 1e-6 & 1e-7 & 5e-8 & 1e-8\\
    $e(q_h^*)$ & 7.16e-2  & 4.76e-2 & 2.39e-2 & 2.04e-2 & 1.98e-2 & 1.23e-1 & 7.76e-2 & 3.58e-2 & 2.29e-2 & 1.54e-2\\
    \midrule
    \multicolumn{1}{c}{}&
    \multicolumn{5}{c}{(c) Example \ref{ex:para}(i)}&\multicolumn{5}{c}{(d) Example \ref{ex:para}(ii)}\\
    \cmidrule(lr){2-6} \cmidrule(lr){7-11}
    $\epsilon$   & 10e-2 & 5e-2 & 1e-2 & 5e-3 &1e-3  & 10e-2 & 5e-2 & 1e-2 & 5e-3 &1e-3\\
    \midrule
    $\gamma_{\theta}$ & 1e-6 & 1e-6 & 1e-6 & 1e-7 & 1e-7 & 1e-7 & 1e-7 & 1e-8 & 1e-8 & 1e-8\\
    $e(q_\theta^*)$ & 3.30e-2 & 3.13e-2 & 1.48e-2 & 1.47e-2 & 1.08e-2 & 6.21e-2 & 4.62e-2 & 2.85e-2 & 2.63e-2 & 2.68e-2\\
    $\gamma_{h}$ & 2e-6 & 1e-6 & 1e-7 & 5e-8 & 1e-8 & 2e-7 & 1e-7 & 1e-8 & 5e-9 & 1e-9\\
    $e(q_h)$ &  5.63e-2 & 4.97e-2 & 1.58e-2 & 1.53e-2 & 1.21e-2 &  7.18e-2 & 4.70e-2 & 2.08e-2 & 1.81e-2 & 1.80e-2\\
    \midrule
    \multicolumn{1}{c}{}&
    \multicolumn{5}{c}{(e) Example \ref{ex:1D_elliptic_subdomain}}\\
    \cmidrule(lr){2-6}
    $\epsilon$   & 10e-2 & 5e-2 & 1e-2 & 5e-3 &1e-3\\
    \cmidrule(lr){1-6}
        $\gamma_{\theta}$ & 1e-6 & 1e-6 & 1e-6 & 1e-6 & 1e-6 \\
    $e(q_\theta^*)$ & 2.92e-2 & 2.25e-2 & 1.43e-2 & 1.92e-2 & 1.36e-2  \\
    $\gamma_{h}$ & 2e-6 & 1e-6 & 1e-7 & 5e-8 & 1e-8\\
    $e(q_h)$ &  6.15e-2 & 4.09e-2 & 3.01e-2 & 2.20e-2 & 1.52e-2\\
    \bottomrule
    \end{tabular}
\end{table}

The first two examples are about the inverse problem in the elliptic case.
\begin{example}\label{ex:ell}
\begin{itemize}
    \item[(i)] $\Omega=(0,1)$,  $q^{\dag}(x)=2+\sin(2\pi x)$ and  $f\equiv 10$.
\item[(ii)] $\Omega=(0,1)^2$,  $q^{\dag}(x_1,x_2)=2+\sin(2\pi x_1)\sin(2\pi x_2)$, $u_0(x_1,x_2)=4x_1(1-x_2)$ and  $f\equiv 10$. 
\end{itemize}
\end{example}

In the ADAM optimizer, the hybrid scheme employs a learning rate $\text{1e-3}$ and $\text{1e-2}$ for cases (i) and (ii), respectively.
The reconstructions for case (i) in Fig. \ref{Fig:elliptic_1D} show that the hybrid approach is more accurate than the pure FEM, although visually they are largely comparable, consistent with the prior observation \cite{berg2021neural}. This is also confirmed by the relative errors in Table \ref{tab:exam}(a). These results clearly show the influence of the discretization scheme on numerical inversion. The excellent performance of the hybrid method might be attributed to the strong implicit smoothness prior imposed by NNs, which strongly favors smooth solutions \cite{Rahaman:2019}, when compared with that by the FEM basis.

\begin{figure}[h!]
\centering
\begin{tabular}{ccc}
\includegraphics[width=0.3\textwidth]{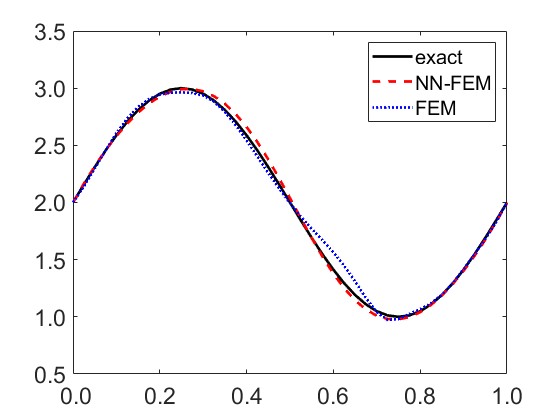} &
\includegraphics[width=0.3\textwidth]{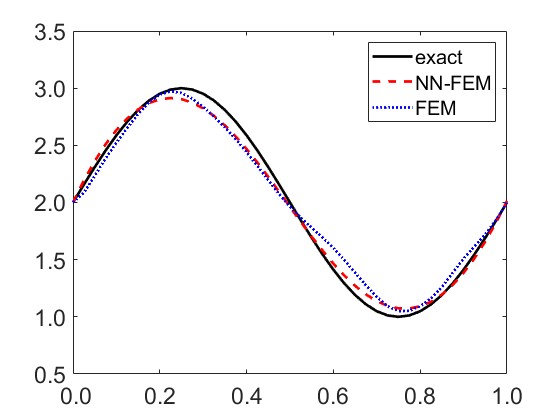} &
\includegraphics[width=0.3\textwidth]{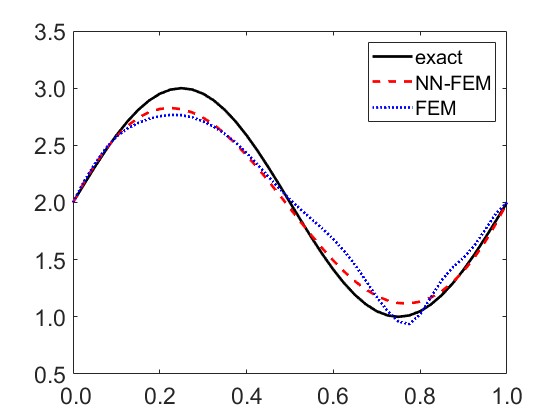} \\
(a) $1\%$ & (b) $5\%$ & (c) $10\%$
\end{tabular}
\caption{The reconstructions for Example \ref{ex:ell}(i) at three noise levels by the hybrid approach and pure FEM.}
\label{Fig:elliptic_1D}
\end{figure}

To gain further insights, we examine the change of the loss during the training process in Fig. \ref{Fig:elliptic_1D_loss}. The plots are for two cases: the 1-32-32-1 architecture with different noise levels to study the impact of data noise, and three architectures: i.e., 1-16-16-1, 1-32-32-1 and 1-32-32-32-1 (at a noise level $5\%$) to study the impact of the architectural choice. During the training, the loss $J$ first decreases only slowly, exhibiting a plateau phenomenon, and then it experiences a rapid decreasing period, after which it almost stagnates and oscillates a little bit. This pattern is consistently observed for all the considered noise levels. The origin of the plateau remains elusive; see \cite{AinsworthShin:2021} for an interesting investigation of the phenomenon for gradient descent on ReLU networks. The evolution of the relative error shows a similar behavior: it first decreases slowly, then enjoys a fast decay and finally tends to be nearly steady.

\begin{figure}[h!]
\centering
\subfigure[loss v.s. noise level]{
\includegraphics[width=0.4\textwidth]{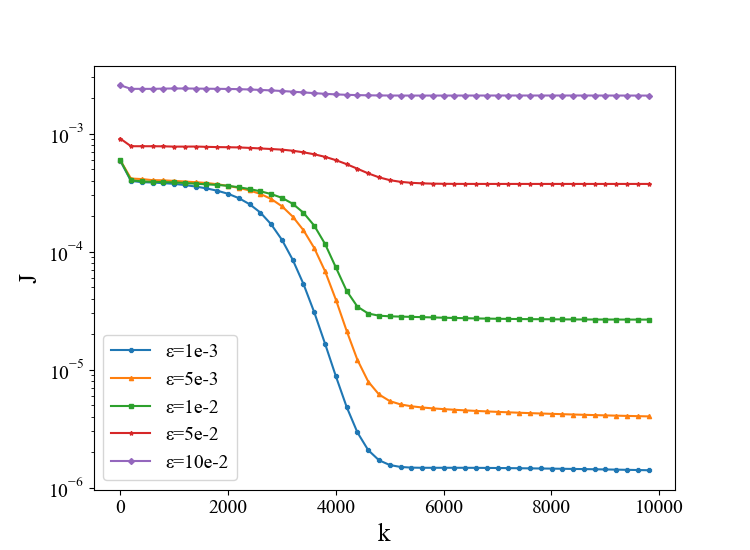}}
\subfigure[relative error v.s. noise level]{
\includegraphics[width=0.4\textwidth]{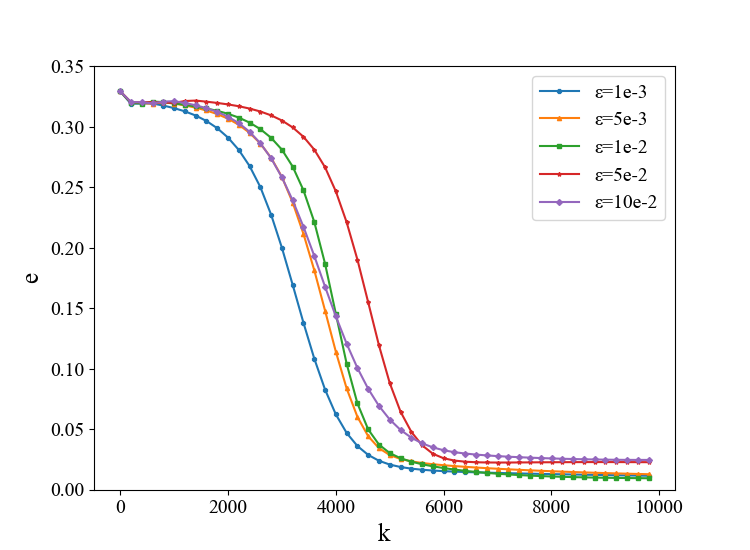}}
\subfigure[loss v.s. NN architecture]{
\includegraphics[width=0.4\textwidth]{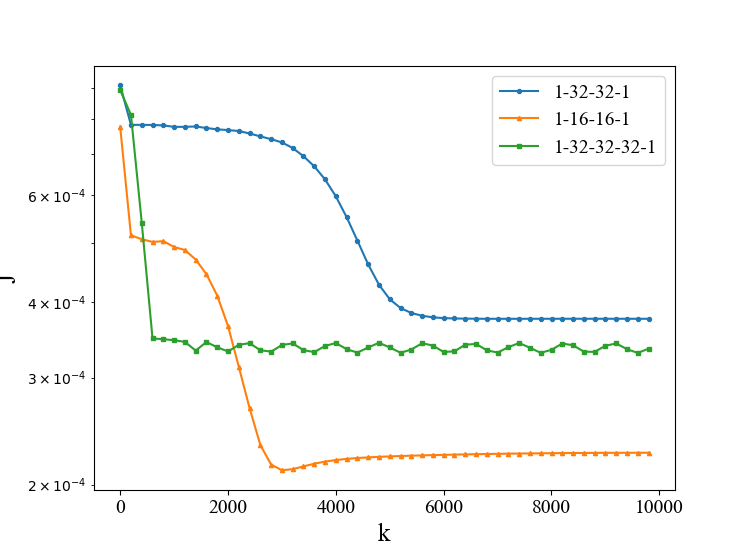}}
\subfigure[relative error v.s. NN architecture]{
\includegraphics[width=0.4\textwidth]{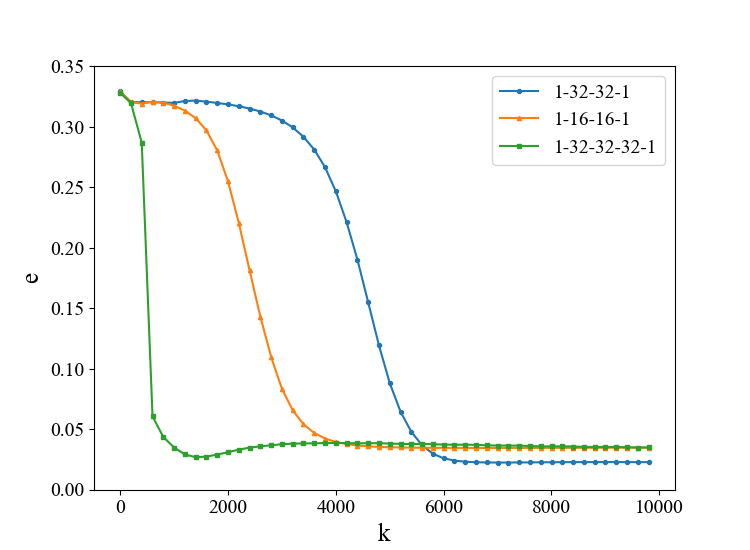}}
\caption{The variation of the loss $J$ and error $e$ during the training for Example \ref{ex:ell}(i) at different noise levels and NN architectures.}
\label{Fig:elliptic_1D_loss}
\end{figure}

Now we examine the influence of quadrature error, by varying the quadrature level $n$ over the set $\{0,1,\cdots,5\}$. This is carried out on two settings with $1\%$ noise: (i) the standard setting as before, and (ii) the setting with the architecture 1-128-128-128-1, a mesh size $h=1/40$, $\gamma=\text{1e-6}$, and a learning rate 1e-3. The architecture in the latter is far bigger, and hence, according to the error estimate in Theorem \ref{thm:error-ellip-q}, the problem becomes more challenging and may require more quadrature points to deliver quality reconstructions. The numerical results are given in Fig. \ref{Fig:1D_elliptic_quad}. It is observed that the relative error $e$ does decay slightly when using more quadrature points but the influence is very minor. Hence, the error bound in Theorem \ref{thm:error-ellip-q} might be overly pessimistic in terms of the quadrature error. In the rest of the experiments, we do not increase the quadrature level.

\begin{figure}[h!]
\centering
\subfigure[Example \ref{ex:ell}(i)]{
\includegraphics[width=0.4\textwidth]{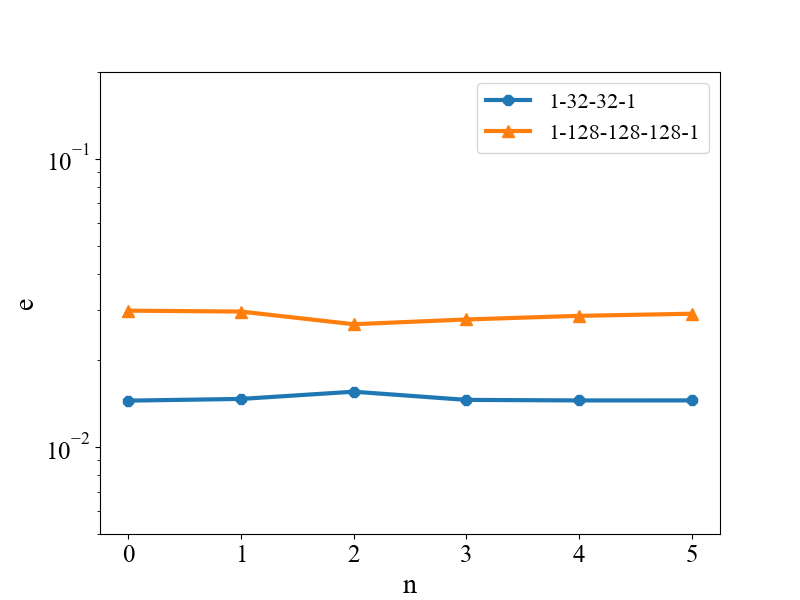}}
\subfigure[Example \ref{ex:para}(i)]{
\includegraphics[width=0.4\textwidth]{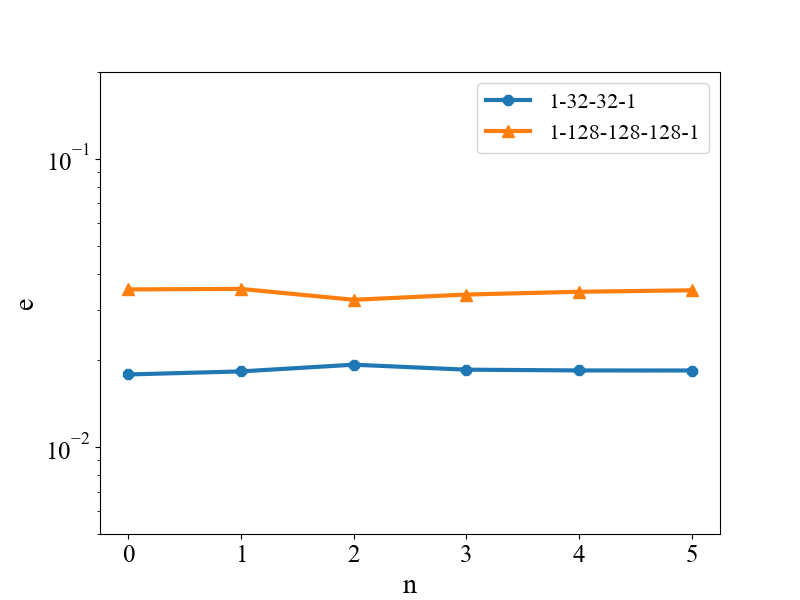}}
\caption{The relative errors for Examples \ref{ex:ell}(i)  and \ref{ex:para}(i) versus quadrature level and NN architectures, at a noise level $1\%$.}
\label{Fig:1D_elliptic_quad}
\end{figure}

In case (ii), the reconstructions by the hybrid approach is slightly more accurate than that by the pure FEM, when the data is highly noisy; see Fig. \ref{Fig:elliptic_2D} and Table \ref{tab:exam}(b). When the data is very accurate, the hybrid approach is actually slightly less accurate. This is attributed to the complex optimization issue: the loss is highly nonconvex in the NN parameters, and its landscape is very complicated, which may prevent the ADAM optimizer from finding a global minimizer. Fig. \ref{Fig:elliptic_2D_loss} shows the evolution of the loss $J$ and relative error $e$ in the two settings during the training process: (i) the 2-32-32-1 architecture with noise level varying form  $0.1\%$ to $10\%$ and (ii) with a fixed $1\%$ noise level, on three NNs, i.e., 2-16-16-1, 2-32-32-1, and 2-32-32-32-1. The results show a similar behavior as for case (i): the convergence curve shows fast convergence only after an initial plateau (of length about 5000 iterations). This may indicate the need of a better initialization strategy for the NN parameters in order to shorten the plateau length (and thus faster convergence).

\begin{figure}[h!]
\centering
\subfigure[exact]{
\includegraphics[width=0.3\textwidth]{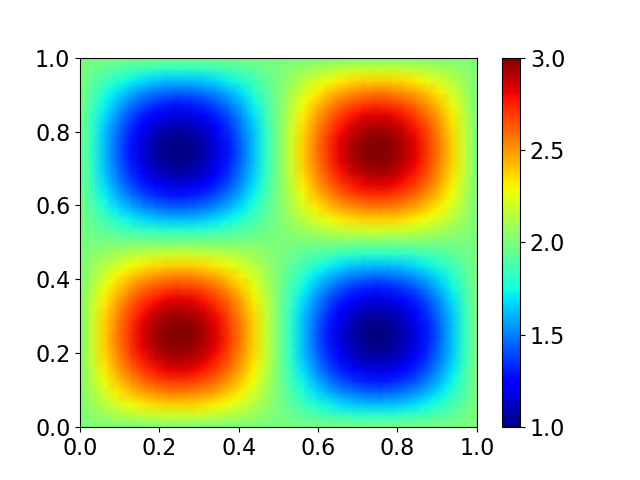}}
\subfigure[$1\%$ noise]{
\includegraphics[width=0.3\textwidth]{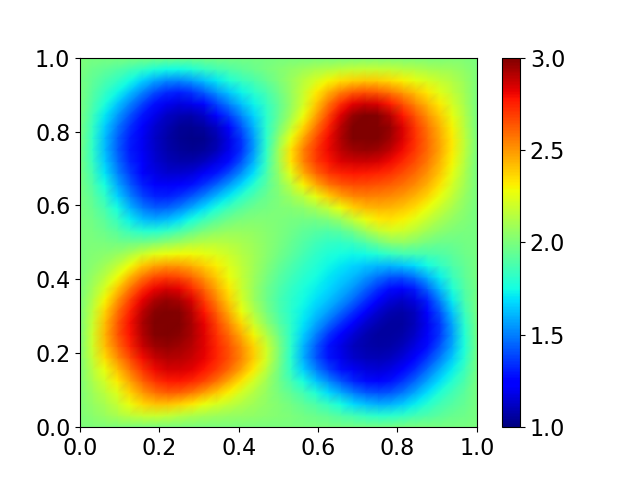}}
\subfigure[$5\%$ noise]{
\includegraphics[width=0.3\textwidth]{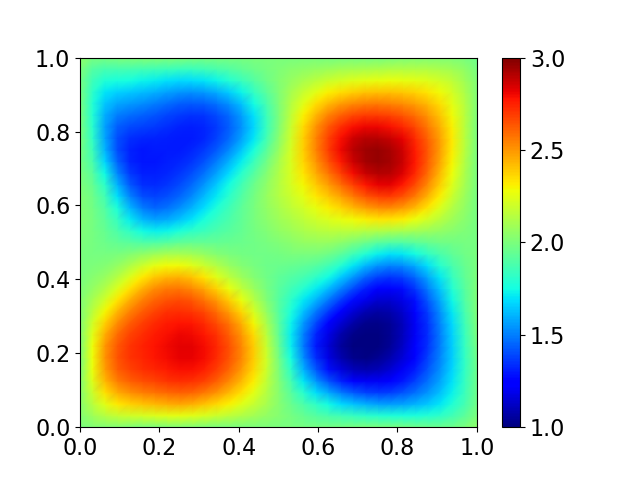}}

\subfigure[exact]{
\includegraphics[width=0.295\textwidth]{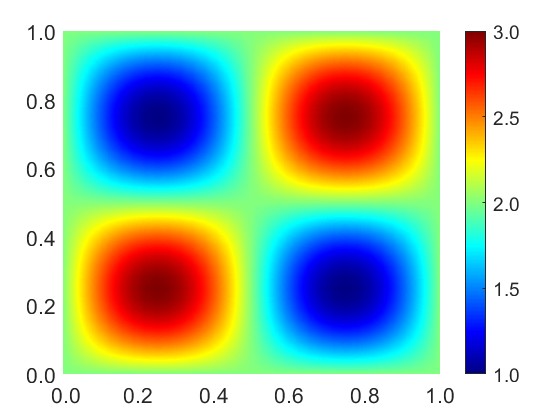}}
\subfigure[$1\%$ noise]{
\includegraphics[width=0.295\textwidth]{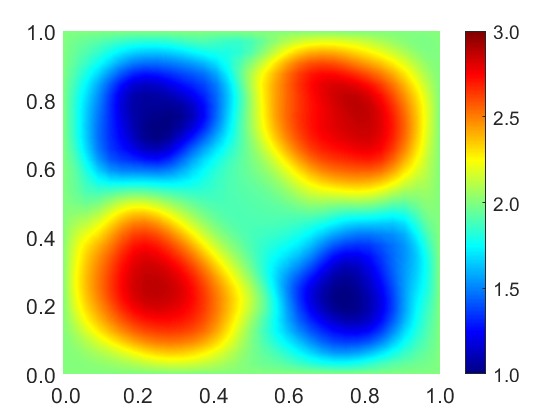}}
\subfigure[$5\%$ noise]{
\includegraphics[width=0.295\textwidth]{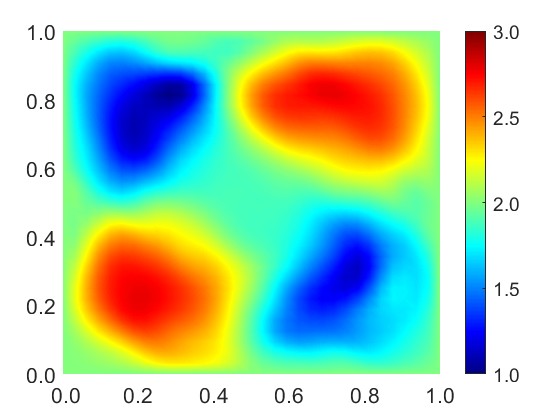}}
\caption{The reconstructions for Example \ref{ex:ell}(ii) at two noise levels with the hybrid method (top) and the pure FEM (bottom).}
\label{Fig:elliptic_2D}
\end{figure}

\begin{figure}[h!]
\centering
\subfigure[loss v.s. noise level]{
\includegraphics[width=0.45\textwidth]{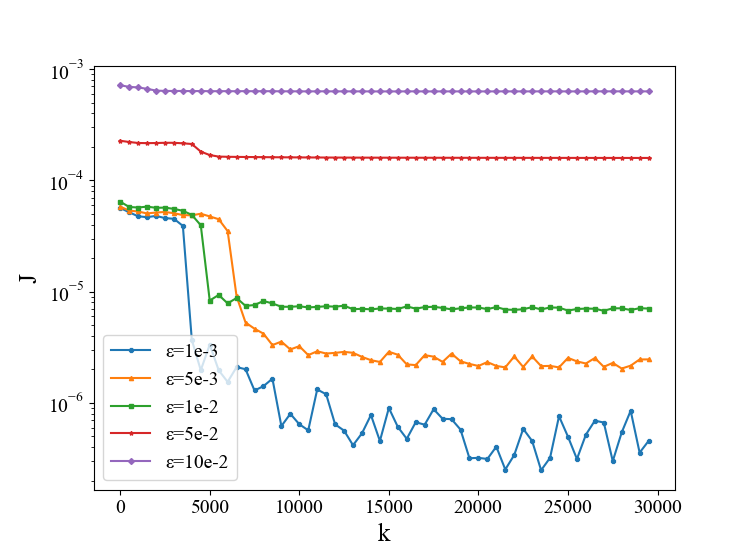}}
\subfigure[relative error v.s. noise level]{
\includegraphics[width=0.45\textwidth]{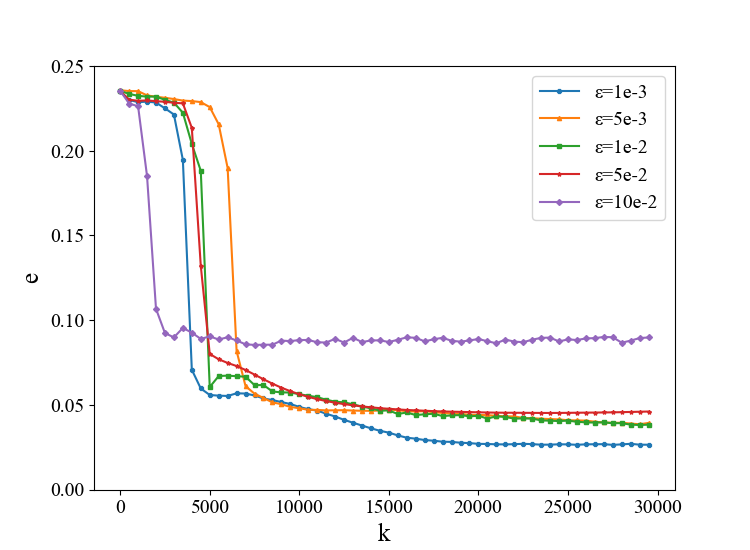}}
\subfigure[loss v.s. NN architecture]{
\includegraphics[width=0.45\textwidth]{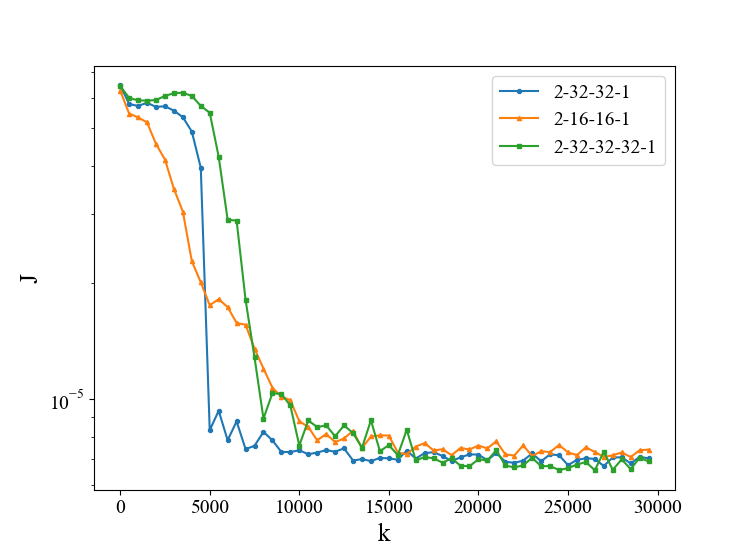}}
\subfigure[relative error v.s. NN architecture]{
\includegraphics[width=0.45\textwidth]{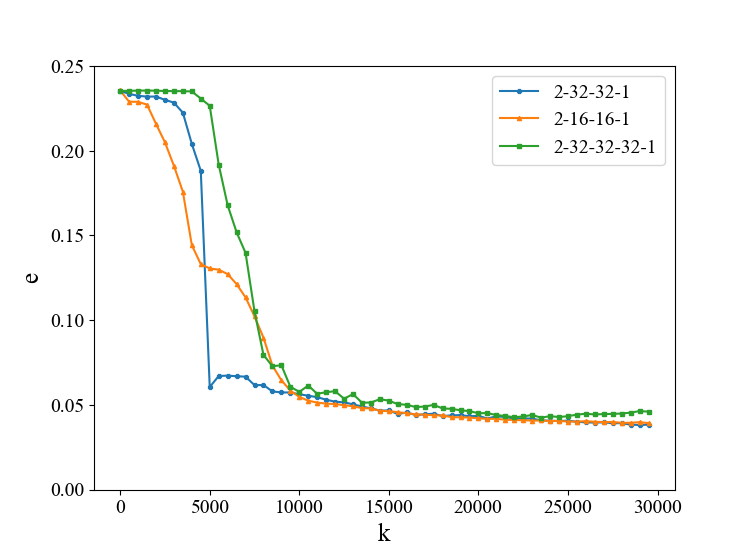}}
\caption{The variation of the loss $J$ and error $e$ during the training for Example \ref{ex:ell}(ii) at different noise levels and NN architecture.}
\label{Fig:elliptic_2D_loss}
\end{figure}

The second set of experiments is for the inverse conductivity problem in the parabolic case.
\begin{example}\label{ex:para}
\begin{itemize}
    \item[(i)] $\Omega=(0,1)$,  $q^{\dag}(x)=2+\sin(2\pi x)$ and  $f(x,t)=10t$, $T_0=0.9$, and $T=1$. 
    \item[(ii)] $\Omega=(0,1)^2$,  $q^{\dag}(x_1,x_2)=2+\sin(2\pi x_1)\sin(2\pi x_2)$, $u_0(x_1,x_2)=4x_1(1-x_1)$ and $f(x_1,x_2,t)= 10t$, $T_0=0.9$ and $T=1$. 
\end{itemize}
\end{example}

In the ADAM optimizer, the hybrid scheme employs a learning rate $\text{1e-3}$ and  $\text{1e-2}$ for cases (i) and (ii), respectively. The numerical results in Figs. \ref{Fig:para_1D} and \ref{Fig:para_2D} (also Tables \ref{tab:exam}(c)--(d)) show similar observations as for the elliptic case: the hybrid approach appears to more accurate for highly noisy data. Likewise, the influence of the quadrature error on the reconstruction eerror $e$ is again very mild, cf. Fig. \ref{Fig:1D_elliptic_quad}(b).

\begin{figure}[h!]
\centering
\begin{tabular}{ccc}
\includegraphics[width=0.3\textwidth]{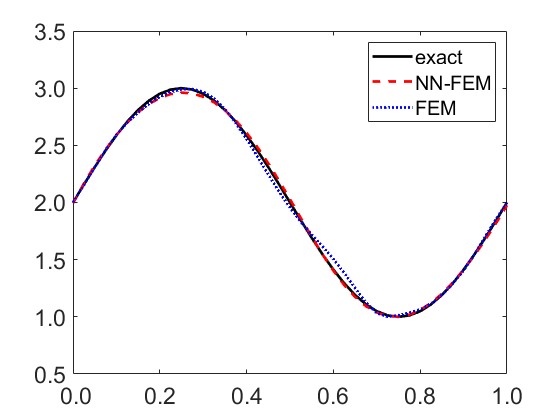}&
\includegraphics[width=0.3\textwidth]{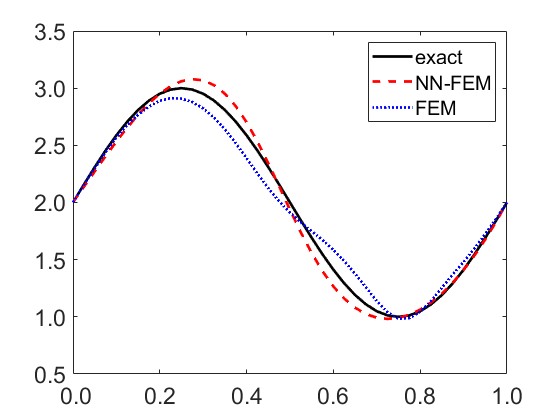}&
\includegraphics[width=0.3\textwidth]{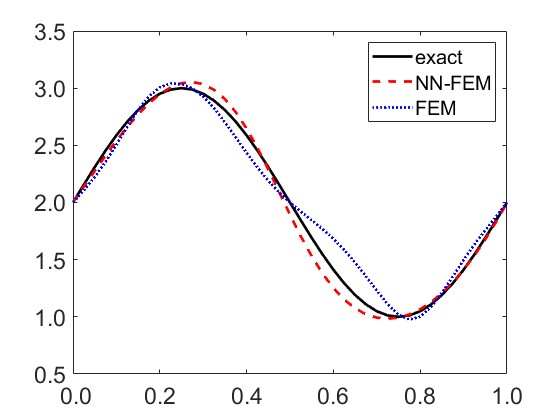}\\
(a) $1\%$ noise & (b) $5\%$ noise & (c) $10\%$ noise
\end{tabular}
\caption{The reconstructions for Example \ref{ex:para}(i) with three noise levels, obtained by the hybrid method and the pure FEM.}
\label{Fig:para_1D}
\end{figure}

\begin{figure}[h!]
\centering
\begin{tabular}{ccc}
\includegraphics[width=0.3\textwidth]{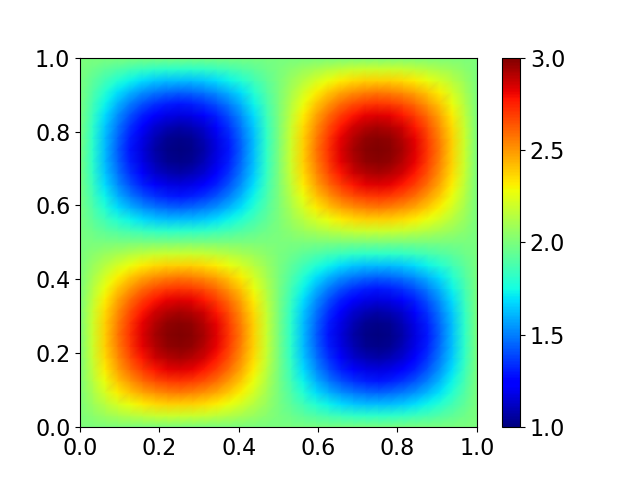}&
\includegraphics[width=0.3\textwidth]{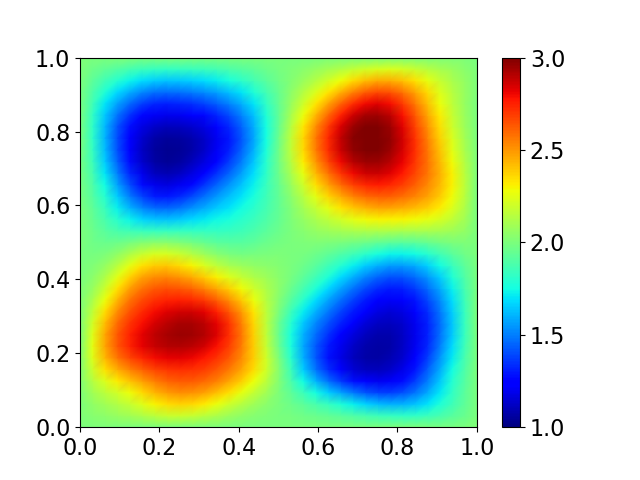}&
\includegraphics[width=0.3\textwidth]{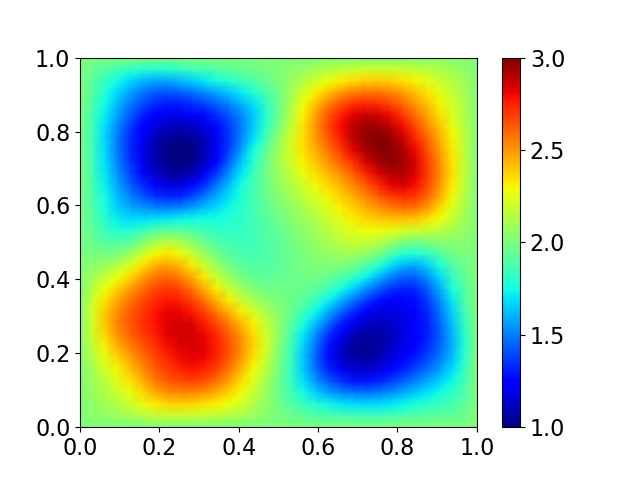}\\
\includegraphics[width=0.295\textwidth]{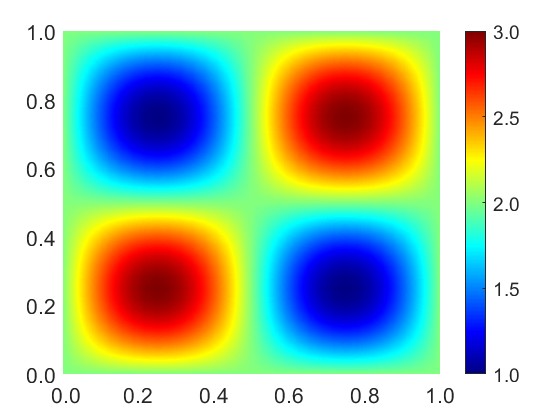} &
\includegraphics[width=0.295\textwidth]{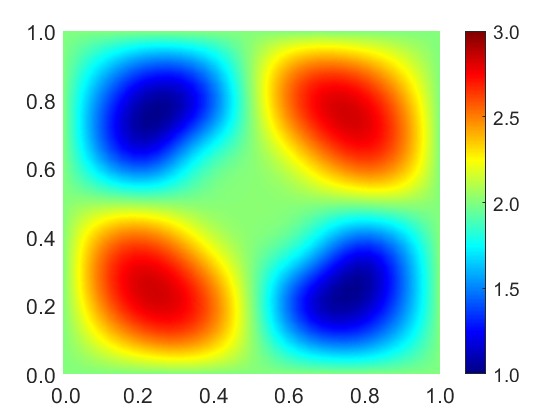} &
\includegraphics[width=0.295\textwidth]{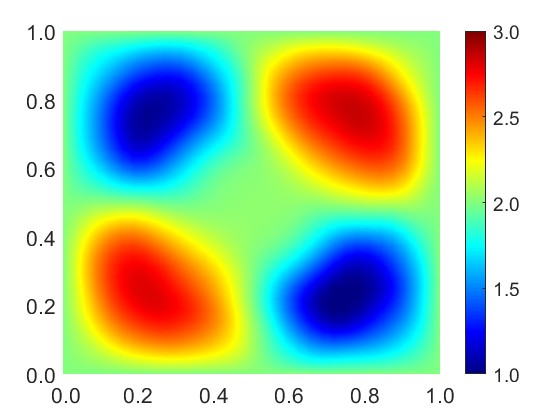}\\
(a) exact & (b) $1\%$ noise & (c) $5\%$ noise
\end{tabular}
\caption{The reconstructions for Example \ref{ex:para}(ii) at two noise levels, by the hybrid method (top) and the pure FEM (bottom).}
\label{Fig:para_2D}
\end{figure}

The last example is about partial interior data (on a subdomain $\omega\subset \Omega$).
\begin{example}\label{ex:1D_elliptic_subdomain}
$\Omega=(0,1)$, $q^{\dag}(x)=2+10(1-x)x^2$ and $f\equiv 10$, $\omega=(0.3,0.7)$.
\end{example}

For the hybrid inversion, we employ a learning rate $\text{1e-3}$. The numerical results are presented in Fig. \ref{Fig:1D_elliptic_subdomain}; see also Table \ref{tab:exam}(e) for the relative errors. Due to the availability of the only partial interior data, the problem is far more ill-posed. It is observed that the reconstructions by the hybrid approach is more accurate than that by the pure FEM, indicating the high robustness of the hybrid approach for more challenging inverse problems.

\begin{figure}[h!]
\centering
\begin{tabular}{ccc}
\includegraphics[width=0.3\textwidth]{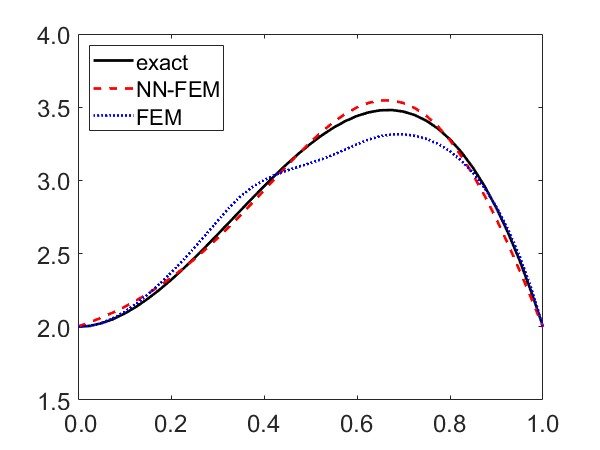}&
\includegraphics[width=0.3\textwidth]{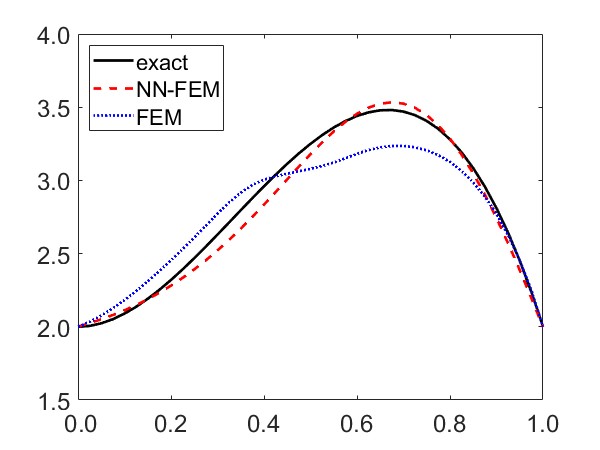}&
\includegraphics[width=0.3\textwidth]{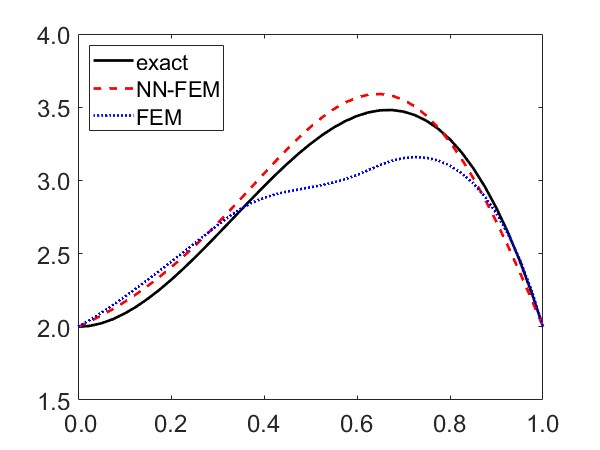}\\
(a)  $1\%$ noise & (b) $5\%$ noise & (c) $10\%$ noise
\end{tabular}
\caption{The numerical reconstructions for Example \ref{ex:1D_elliptic_subdomain} at three noise levels, obtained with the hybrid method and the pure FEM.}
\label{Fig:1D_elliptic_subdomain}
\end{figure}

\bibliographystyle{abbrv}

\end{document}